\documentclass[11pt,fleqn]{article}

\usepackage{amscd, amsfonts, amsthm, amssymb, mathrsfs,
amsthm,bbm,makeidx} 

\usepackage{graphics,color}
\usepackage[all]{xy}
\usepackage{amssymb}
\usepackage{latexsym}
\usepackage{amsmath}
\numberwithin{equation}{section}

\newtheorem{Satz}{Satz}[section]

\newtheorem{prop}[Satz]{Proposition}

\newtheorem{definition}[Satz]{Definition}
\newtheorem{deflem}[Satz]{Definition/Lemma}
\newtheorem{remark}[Satz]{Remark}
\newtheorem{cor}[Satz]{Corollary}
\newtheorem{lemma}[Satz]{Lemma}
\newtheorem{example}[Satz]{Example}

\newtheorem*{theorem}{Theorem}

\newcommand{\Cone}{{\mathrm{Cone}}}
\newcommand{\sing}{\mathrm{sing}}

\newcommand{\ra}{\rightarrow}

\newcommand{\DD}{{\cal{D}}}
\newcommand{\FFF}{{\cal{S}}}
\newcommand{\D}{{\mathbf{D}}}
\newcommand{\FF}{{\cal{F}}}

\newcommand{\ess}{{\mathrm{ess}}}

\newcommand{\RRR}{\mathbb{R}}
\newcommand{\NNN}{\mathbb{N}}
\newcommand{\TTT }{\mathbb{T}}
\newcommand{\CCC }{{\cal{C}}}
\newcommand{\EEE }{{\cal{E}}}
\newcommand{\DDD }{{\bf{\Delta}}}

\newcommand{\Dcal }{{\cal{D}}}
\newcommand{\Rcal }{{\cal{R}}}

\newcommand{\HHH}{{\cal{H}}}

\newcommand{\LLL}{{\mathrm{L}}}
\newcommand{\cone}{{\mathrm{cone}}}
\newcommand{\conf}{{\mathrm{model}}}
\newcommand{\dom}{\mathop{\mathrm{dom}}}
\newcommand{\range}{\mathop{\mathrm{range}}}
\newcommand{\im}{\mathop{\mathrm{im}}}

\newcommand{\Crit}{\mathrm{Crit}}

\newcommand{\spa}{\mathop{\mathrm{span}}}
\newcommand{\spec}{\mathrm{spec}}
\newcommand{\supp}{\mathop{\mathrm{supp}}}

\newcommand{\Lie}{{\cal{L}}}

\newcommand{\PPP}{{\mathbb{P}}}

\newcommand{\ev}{\mathrm{ev}}
\newcommand{\odd}{\mathrm{odd}}

\newcommand{\C}{{\mathbb{C}}}

\def\haken{\mathbin{\vrule height0.6pt width0.6em \vrule height0.6em}}

\newcommand{\norm}[1]{\lVert #1 \rVert}
\newcommand{\abs}[1]{| #1 |}

\begin{document}

\title{The Witten deformation for even dimensional conformally conic manifolds}

\author{Ursula Ludwig}

 \maketitle

\begin{abstract}
The goal of  this article is to generalise the Witten deformation to even dimensional conformally conic  manifolds  $X$ and a  class of functions $f: X \ra \RRR$ called admissible Morse functions. We get  Morse inequalities relating the  $\LLL^2$-Betti numbers of $X$ with the number of critical points of the function $f$. Hereby the contribution of a singular point $p$ of $X$ to the Morse inequalities can be expressed in terms of  the intersection cohomology of the local Morse data of $f$ at $p$. The definition of an admissible Morse function is inspired by stratified Morse theory as developed by Goresky and MacPherson. 
\end{abstract}
MSC-class  35A20 (Primary),  57R70 (Secondary)

\section{Introduction}

The Witten deformation is a method proposed in \cite{witten} by Witten  which, given a smooth Morse function $f:M \ra \RRR$ on a smooth compact  Riemannian manifold $M$, leads to an analytical proof of the Morse inequalities. A  rigorous  account of the analytic proof of the Morse inequalities using semi-classical analysis has been   done in  \cite{hs4}. It was used in \cite{bismutzhang} to give an extension of a theorem of Cheeger and M\"uller on the relation between the Ray-Singer analytic torsion and the Reidemeister torsion.

This article generalises the Witten deformation for conformally conic manifolds $(X,g)$ of even dimension $\dim X = 2\nu$ and a class of functions which are called  admissible Morse functions.   Conformally conic manifolds  generalise Riemannian manifolds with cone-like singularities (see \cite{leschcone}).  The results presented here then also hold in particular for spaces with cone-like singularities. The definition of an admissible Morse function introduced in this article is motivated by the non-degeneracy condition of stratified Morse functions in the sense of stratified Morse theory developed by Goresky and MacPherson \cite{goresky}. However note that the settings in \cite{goresky} and here are slightly different. The spaces considered in \cite{goresky} are Whitney stratified spaces (embedded in some $\RRR^N$) and the non-degeneracy condition for a stratified Morse function is expressed in terms of the generalised tangent spaces for the Whitney stratification. The main reason for us to work on conformally conic manifolds is  that for these spaces  $\LLL^2$-techniques are well understood, and we thus have good tools for an analytic proof of the Morse inequalities. 


In the rest of this introduction we will explain how to adapt Witten's method to the situation described above, state the main results and explain shortly the idea of proof.

The main principle in Morse theory is to give a relation between a "local datum" of the Morse function, namely its critical points, and a "global topological datum" of the space.  For smooth manifolds the latter  is the singular cohomology of the manifold. In the presence of singularities the topological invariant of interest is the so called intersection cohomology.

For conformally conic manifolds intersection cohomology can be analytically expressed by using the complex of $\LLL ^2$-forms: Let us denote by $\Sigma$ the singular set of $X$. Let $(\Omega _0^*(X \setminus \Sigma ), d)$ be the de Rham complex of differential forms with compact supports. For conformally conic manifolds the   elliptic complex $(\Omega _0^*(X \setminus \Sigma ), d)$ admits a unique  extension into a Hilbert complex  $(\CCC, d, \langle \ , \  \rangle ) $  in the Hilbert space of square integrable forms equipped with the $\LLL^2$-metric 
\begin{equation*} \langle  \alpha , \beta \rangle  := \int _{X \setminus \Sigma} \alpha \wedge * \beta . \end{equation*}
 The \emph{$\LLL^2$-cohomology} of $X$, denoted by $ H_{(2)} ^i (X)$,  is defined as   the cohomology of this Hilbert complex.   (See Section 2 for details. Note that  in this article the language of Hilbert complexes, as introduced in \cite{hilbert} is used.)

Witten's idea for an analytic proof of the Morse inequalities on a smooth compact manifold consists in the deformation of the de Rham complex by means of a smooth Morse function (see \cite{witten}, \cite{hs4}). In the presence of singularities we deform  the complex of $\LLL^2$-forms instead. We use  an admissible Morse function $f: X \ra \RRR$ for the deformation (see Definition \ref{defadmissible}). In particular we deform the complex $(\Omega _0^*(X \setminus \Sigma ), d)$ into 
\begin{equation}\label{deformedcomplex} 0 \ra \Omega ^0 _0(X \setminus \Sigma ) \xrightarrow{d_t} \ldots \xrightarrow{d_t } \Omega ^{2 \nu}  _0(X \setminus \Sigma  ) \ra 0 , \end{equation}
 where   $d_t = e^{-ft}d e^{ft}$; here $t \in (0, \infty)$ is the deformation parameter. One can show that  the deformed complex also admits a unique extension into a Hilbert complex, which is denoted by $(\CCC_t, d_t, \langle \ , \  \rangle )$. The map $\omega \ra e^{-tf} \omega$ yields an isomorphism of the two complexes $(\CCC , d , \langle \ , \  \rangle)$ and $( \CCC _t, d_t , \langle \ , \  \rangle)$. Therefore the cohomology of the deformed complex is also isomorphic to the $\LLL^2$-cohomology of $X$, {\it i.e.} $H^i(\CCC_t, d_t, \langle \ , \  \rangle) \simeq H^i _{(2)} (X)$.

Let us denote by $\delta _t$ the adjoint of $d_t$ with respect to the $\LLL^2$-metric. The Witten Laplacian is defined as the  Laplacian associated to the Hilbert complex $(\CCC _t, d_t , \langle \ , \  \rangle)$, {\it  i.e.}
\begin{equation}
\begin{split}
& \Delta _t =   d_t \delta _t + \delta _t d _ t , \\
& \dom (\Delta _t ) = \{\omega  \mid \omega,  d_t \omega , \delta _t \omega , d_t \delta _t \omega, \delta _t d_t \omega \in \LLL^2 (\Lambda ^*(T^*(X \setminus \Sigma)))\}.\end{split}
\end{equation} Note that in the presence of singularities $\Delta _{t} \restriction _{\Omega ^*_0 ( X\setminus \Sigma)}$ is not an essentially self-adjoint operator and therefore we have to specify the domain of the Witten Laplacian carefully. The Witten Laplacian is a non-negative, self-adjoint operator with discrete spectrum. Hodge theory is still valid for the deformed complex, {\it i.e.}
\begin{equation}\label{hodgeeq} \ker (\Delta _{t}^{(i)})  \simeq H^i(\CCC_t,d_t, \langle \ , \  \rangle ), \quad i = 0, \ldots, 2 \nu ,\end{equation}
where $\Delta _t^{(i)}$ denotes the restriction of $\Delta _t$ acting on $i$-forms. 

 The advantage of the deformed complex $(\CCC _t, d_t , \langle \  , \  \rangle )$ compared to the initial complex $( \CCC , d , \langle \  , \  \rangle )$ is that the spectrum of the Witten Laplacian has nice properties for large parameters $t$. For an admissible Morse function $f$ the restriction  $f _{\mid X  \setminus \Sigma}$ is a Morse function in the smooth sense and we denote by $c_i (f_{\mid X \setminus \Sigma})$        the number of critical points of $f_{\mid X \setminus \Sigma}$ of index $i$.  

\begin{theorem}[Spectral gap theorem]\label{thmspectralgap} 
\begin{enumerate}
\item Let $(X,g)$ be an even dimensional conformally conic Riemannian manifold  and let $f:X \ra \RRR$ be an admissible  Morse function. Then there exist constants $C_1, C_2, C_3 >0$ and $t_0>0$ depending on $X$ and $f$ such that for any $t >t_0$,
\begin{equation*} \spec (\Delta _{t}) \cap (C_1e^{-C_2t}, C_3t) = \emptyset.\end{equation*}
\item  Let us denote by $(\FFF_t, d_t, \langle \ , \  \rangle  )$ the subcomplex of $(\CCC _t, d_t,  \langle \ , \  \rangle )$ generated by all eigenforms of the Witten Laplacian $\Delta _t$ to eigenvalues in $[0,1]$.  Then, for $t \geq t_0$,
\begin{equation}\label{mpi1} \begin{array} {ll}  \dim \FFF^{i}_t =  c_i (f_{\mid X \setminus \Sigma}) + \displaystyle \sum _{ p \in \Sigma} m_p^i =:c_i(f),  \end{array} \end{equation} where $m_p^i$ is the contribution of the singular point $p$  to $c_i(f)$ explained in more detail below.
\end{enumerate}
   \end{theorem}

The precise definition of the $m_p^i$ is given in Section 3. While it is rather technical  we shortly explain it here also for completeness:  For a point $p \in \Sigma$ and $\epsilon >0$ let us denote by $ B _{\epsilon} (p)$ the closed $\epsilon$-ball around  $p$. Then, for $\epsilon$ small enough, $ B _{\epsilon} (p)$ is homeomorphic to the closed cone $L_p \times [0, 1 ] / {0 } \times L _p$. The closed $2 {\nu} -1$-dimensional manifold $L_p$ is  called the link of the singularity. Let us choose $0 < \delta << \epsilon$. The local Morse data of $f$ at $p \in \Sigma$ reduce in this case to the normal local Morse data and  are defined as the pair of spaces
\begin{equation} (M_p, l_p^-) := \Big ( B_{\epsilon}(p) \cap f^{-1} ([f(p) - \delta, f(p) + \delta ]) , B_{\epsilon}(p) \cap f^{-1} ( f(p) - \delta ) \Big ) .\end{equation} The pair $(M_p, l_p^-) $ is independent of the  choice of $0 < \delta << \epsilon$.   As in stratified Morse theory (see \cite{goresky}, pg. 66) we  call
$ l^- _p:= B_{\epsilon}(p) \cap f^{-1} ( f(p) - \delta ) $ the lower halflink.  It is not difficult to see that the pair  $(M_p, l_p^-)$  is homeomorphic to $(B_{\epsilon} (p)  , l^-_p)$. The contribution of the singular point $p$ to the Morse inequalities  is related to the local Morse data as follows 
\begin{equation}\label{mpidim} m_p^i = \dim IH ^i (M_p, l_p^-) = \dim IH ^i(B_{\epsilon} (p)  , l^-_p),  \end{equation} where $ IH^*$ denotes the intersection cohomology (with closed support).

   As usual  the following Morse inequalities follow from the spectral gap theorem and the Hodge theory  for the deformed complex \eqref{hodgeeq}  by a simple algebraic argument

\begin{cor}\label{cormorse} In the situation of Theorem \ref{thmspectralgap}
\begin{equation}\label{morseinequ}
\begin{split} & \sum _{i = 0} ^ {k} (-1) ^{k-i}c_i (f)  \geq \sum _{i = 0} ^ {k} (-1) ^{k-i} b_i^{(2)} (X), \text{ for all }  0 \leq k <  2 \nu,
\\
&  \sum _{i = 0} ^ {2\nu} (-1) ^{i}c_i (f)  = \sum _{i = 0} ^ {2 \nu} (-1) ^{i} b_i^{(2)} (X), \end{split} \end{equation} where $b_i^{(2)} (X):= \dim H_{(2)} ^i (X)$ denote the $\LLL^2$-Betti numbers of $X$.
\end{cor}

Note that a particular example of the situation treated here is the case where  $X$ is a complex cone and $f: X \ra \RRR$ is a stratified Morse function in the sense of \cite{goresky}. Then  the Morse inequalities in \eqref{morseinequ} give back the stratified Morse inequalities in \cite{goresky}.

   The key step in the proof of the spectral gap theorem is to construct a local model operator $\DDD _{t,p}$ for the Witten Laplacian near a singular point $p$ of $X$ and to show a local spectral gap theorem for  $\DDD _{t,p}$.

In a previous paper \cite{ul} (see also \cite{ulcurve})  the Witten deformation for singular complex curves $C \subset \PPP ^ n( \C )$ equipped with a stratified Morse function in the sense of the theory in \cite{goresky} has been discussed. There the local model has a particularly simple form which has  been treated by an explicit computation. In  the higher dimensional situation the local model is more complicated and the local spectral gap theorem is shown using perturbation techniques for regular singular operators as in \cite{leschcurve}. A cone construction adapted from \cite{braverman} (see also \cite{farber}) is used to prove the relation to the geometry.

These notes are organised as follows: In Section 2 we shortly recall the basic facts on the $\LLL^2$-cohomology of conformally conic manifolds. We also define the class of admissible Morse functions and explain the Witten deformation in this singular context. The results in Section 2 are direct generalisation of the curve case and they also hold mutatis mutandis for the odd dimensional case. We give the proofs in some detail however for convenience of the reader. The main work is  done in Section 3 where we define the local model operator $\DDD _t$ for the Witten Laplacian near a singular point of $X$ and show the local spectral gap theorem.  In Section 4 an outline of the proof of the spectral gap theorem (Theorem \ref{thmspectralgap}) and the Morse inequalities (Corollary \ref{cormorse}) is given. Once the local situation near singular points of $X$ is understood the proof of the spectral gap theorem is a direct generalisation of the proof in the smooth case (here we follow the proof in \cite{bismutlebeau}, Section 9). 

The results presented here have been announced in \cite{ulcras3}.

\section{The Witten deformation for conformally conic manifolds  and admissible functions}\label{sectionhodget}

Let $X$ be a topological space, which moreover is a smooth manifold of dimension $\dim X = 2 \nu$ outside a set of isolated singularities $\Sigma := \{p _1, \ldots, p_N\}$. Let $g $ be a Riemannian metric on $X \setminus \Sigma$. We call the pair  $(X,g)$ a conformally conic manifold (see \cite{leschcone}) if
\begin{itemize} 
\item For each $p \in \Sigma$ there exists an open neighbourhood $U_p$ in $X$ such that $ ( U_p \setminus \{p\}, g _{\mid U_p \setminus \{p\}})  $  is isometric to $((0, \epsilon) \times L _ p , g(r))$. Hereby $L_p$ is a closed manifold of  dimension $\dim (L_p) = 2\nu-1=:n$ called the link of the singular point $p$. Moreover $g(r) = h(r) ^2 (dr^2 \oplus r^2 g_{L _p }(r))$, where $g_{L_p}(r)$ is a family of metrics on $L_p$, smooth in $(0, \epsilon)$ and continuous in $[0, \epsilon) $, and $h \in C^{\infty}((0, \epsilon) \times L_p)$ satisfies \begin{equation} \sup _{\varphi  \in L_p} | (r \partial _r )^j (r^{-c} h(r,\varphi ) -1) | = O (r^{\delta}) \text{ as } r \ra 0, \quad j=0,1 , \end{equation}
and    \begin{equation}\sup _{\varphi  \in L_p} \norm{ h(r,\varphi )^{-1} d_{L _p} h(r,\varphi )} _{T_{\varphi }^*L, g_{L_p}(r)} = O (r^{\delta} ) \text{ as } r \ra 0 , \end{equation} 
for some $\delta >0$ and $c>-1$.

\item If we denote by \begin{equation*} g^0:= dr^2 + r^2 g_{L_p}(0), \quad g^1:= h^{-2} g=  dr^2 + r^2 g_{L_p}(r)\end{equation*} and by $\omega^0$, $\omega ^1$ the connection forms of  the Levi-Civita connections for $g^0$, $g^1$, then \begin{equation}\sup _{ \varphi \in L_p} ( |g^1-g^0|_{(r,\varphi )}^0+ r |\omega ^1 - \omega ^0 |_{(r,\varphi )}^0) = O (r^{\delta} ) \text{ as } r \ra 0 , \end{equation}
where $^0$ refers to the metric $g^0$.
\end{itemize}

 We will  also assume for the rest of this paper that $X$ is compact.  Via the transformation $r \mapsto \frac{1}{1+c} r^{c+1}$ the metric $g$ changes into $\tilde g = \tilde h^2 (dr^2 +r^2 \tilde g_{L_p}(r))$ which satisfies all estimates above with $c=0$. From now on we will always assume that we are in this later situation.

Conformally conic Riemannian manifolds generalise Riemannian manifolds with cone-like singularities. In particular a conformally conic manifold  is quasi-isometric to a conic one.  Important examples of conformally conic manifolds are singular complex curves and   complex cones (see \cite{leschcone}).

\

Let $(X,g)$ be a conformally  conic Riemannian manifold. Let $(\Omega _0^*(X \setminus  \Sigma ), d)$ be the de Rham complex of differential forms with compact supports. An {\it ideal boundary condition} (see \cite{cheeger2}, \cite{hilbert}) for the elliptic complex $(\Omega _0^*(X \setminus \Sigma ), d)$ is a choice of closed extensions $D_k $ of $d_k$ in the Hilbert space of square integrable $k$-forms, such that $D_k (\dom (D_k)) \subset \dom (D_{k+1}). $  We then get a Hilbert complex 
\begin{equation} 0 \ra \dom (D_0 ) \xrightarrow{D_0}  \ldots \ldots \xrightarrow{D_{n-1}} \dom (D_n) \ra 0  .\end{equation}
We will abbreviate the term {\it ideal boundary condition} by {\it ibc} in the sequel. The minimal and maximal extension of $d$ 
\begin{equation} d_{\min} := \overline{d} = \text{ closure of }d, \  d_{\max }:= \delta ^*= \text{ adjoint of the formal adjoint } \delta \text{ of } d \end{equation}
are examples of {\it ibc's}. A priori there may be several distinct {\it ibc's}.

 As shown in \cite{cheeger2} in the case of manifolds with cone-like  singularities  we have {\it uniqueness of ideal boundary condition}, {\it i.e.} in that case the minimal and the maximal extension coincide. We also say that the $\LLL^2$-Stokes theorem holds for cone-like singularities. The domains of $d_{\min} $ and $d_{\max}$, and therefore the validity of the $\LLL^2$-Stokes theorem, are quasi-isometry invariants. Therefore also in the case of conformally conic manifold we have a {\it unique ibc}, {\it i.e.} \begin{equation}\label{dmindmax}
d_{k, \min  } = d _{k, \max} \text{ for all } k .
\end{equation} We denote  by $(\CCC, d, \langle \ , \  \rangle)$ the unique extension of the differential complex $(\Omega _0^*(X \setminus \Sigma ), d)$ to a Hilbert complex. The cohomology of this complex is the so called $\LLL^2$-cohomology of $X$: \begin{equation} H_{(2)} ^i (X) := \ker d_{i, \min } / \im d_{i-1, \min }= \ker d_{i, \max } / \im d_{i-1, \max }.\end{equation}

Since the $\LLL^2$-cohomology of $X$ is a quasi-isometry invariant one can compute it also by replacing the conformally conic metric with a conic metric. Therefore it is clear that all $\LLL^2$-cohomology groups $H_{(2)} ^i (X)$ are finite dimensional and the complex $(\CCC, d, \langle \ , \  \rangle)$ is Fredholm. Applying Corollary 2.5 in \cite{hilbert} to the Fredholm complex $(\CCC, d, \langle \ , \  \rangle )$ one deduces that the canonical maps \begin{equation} \ker (\Delta ^{(i)}) \longrightarrow  H_ {(2)} ^i(X), \end{equation}

\noindent are isomorphisms, where by $\Delta$ we denote the Laplacian associated to the Hilbert complex $(\CCC, d, \langle \ , \  \rangle)$ and by $\Delta ^{(i)}$ its restriction to $i$-forms.

The {\it uniqueness of ibc} in the case of conformally conic manifolds has  also been shown  by Br\"uning and Lesch in \cite{leschcone} by a different argument, which will be useful here.

In this section we perform the Witten deformation on the complex of $\LLL^2$-forms  by means of a certain class of functions, which we call admissible functions:

\begin{definition}\label{defadmissible} {\rm
\begin{itemize} \item[a)] A continuous function $f:X \ra \RRR$ is called admissible if its restriction to $X \setminus  \Sigma$ is smooth and moreover locally near a singularity $p \in \Sigma$ the function $f$ has the form \begin{equation}f(r, \varphi) = f(p) +  f_1 (r, \varphi) + f_2(r, \varphi) , \text{ where } f_1 = r h, \  f_2 = O(r^{1+ \delta})   \end{equation}

and $h: L _p \ra \RRR$ is a smooth function. 

\item[b)] An admissible function $f:X \ra \RRR$ is called an admissible Morse function if 
\begin{itemize} \item[(i)] the restriction $f\restriction_{ X \setminus \Sigma} $ is Morse in the smooth sense. \item[(ii)] there exists a neighbourhood $U$ of the singular set $\Sigma$ and a constant $a >0$ such that $\vert \nabla f \vert ^2 \geq a^2$ on $U$. \end{itemize} \end{itemize}}
\end{definition}

\begin{remark} The condition (ii) in the definition of an admissible Morse function is inspired by stratified Morse theory as developed by Goresky and MacPherson \cite{goresky}: A stratified Morse function is not critical in normal directions for any critical point on a lower dimensional stratum. It will enter in our analysis in two different ways which will become more transparent in Section 3:  It allows to show the local spectral gap theorem for the model Witten Laplacian. Moreover we can adapt a method from \cite{braverman}, where the Witten deformation for manifolds with boundary has been  studied, to show the geometric content of $m_p^i$, the contribution of the singular point $p \in \Sigma$ to the Morse inequalities.  \end{remark}

 For the rest of this section let $f: X \ra \RRR$ be an admissible function on the conformally conic manifold $X$.  Let us denote by $(\Omega ^*_0( X \setminus \Sigma)  , d_t, \langle \ , \  \rangle)$  the differential complex of smooth forms with compact supports on $X \setminus \Sigma$, where  $d_ t = e^{-ft}d e^{ft}$ and $\langle \ , \  \rangle$ is the $\LLL^2$-metric, $t \in ( 0, \infty)$ is the deformation parameter. We have two interesting associated Hilbert complexes: the maximal extension  $(\CCC _{t,\max}, d_{t,\max}, \langle \ , \  \rangle) $, defined by 
\begin{equation}d_{t, \max}  =  \text{adjoint of the formal adjoint of } d_t  \text{ w.~r.~t. } \langle \ , \  \rangle, \end{equation}
and the  minimal extension 
$(\CCC _{t,\min}, d_{t,\min}, \langle \ , \  \rangle), $ defined by
\begin{equation} d_{t, \min} =  \text{ closure of } d_t \text{ with respect to } \langle \ , \  \rangle. \end{equation}

\noindent  Denote by $\delta _t$ the formal adjoint of the operator $d_t$ with respect to the metric $ \langle \ , \  \rangle$. And by $\Delta _{t }\restriction _{ \Omega _0^*}= (d_t + \delta _t)^2$ the deformed  Laplacian (acting on smooth compactly supported forms).
 
\begin{lemma}\label{wittenlaplacian}

The following  identities hold for $\omega \in \Omega_0^*(X \setminus  \Sigma)$
\begin{equation}\begin{split}
d_t \omega & =  d \omega + tdf \wedge \omega  ,\\
\delta _t \omega & =   e^{tf} \delta e^{-tf} \omega = \delta \omega +  t \nabla f \haken \omega , \\
\Delta _t \omega& = \Delta \omega  + t (\Lie _{\nabla f}+ \Lie _{\nabla f}^*) \omega + t^2 \vert \nabla f \vert ^2 \omega,\\
\end{split} \end{equation} where we denote by  $\Lie_{\nabla f} = d(\nabla f \haken \ ) + \nabla f\haken d $  the Lie derivative in the direction of the gradient vector field $\nabla f$ and by $ \Lie _{\nabla f}^*$ its adjoint.

\end{lemma}

\begin{proof} See e.g. Prop.~5.4 in \cite{bismutzhang}. \end{proof}
\begin{remark}{\rm Note that the operator $M_f: = \Lie _{\nabla f}+ \Lie _{\nabla f}^*$ is a zeroth order operator.} \end{remark}

Let us  denote by \begin{equation}D_{t}^{\ev} := d_t + \delta _t  : \Omega _0^{\ev} (X \setminus \Sigma ) \longrightarrow  \Omega _0^{\odd} (X \setminus \Sigma)\end{equation} and by \begin{equation}D_{t}^{\odd} := d_t + \delta _t  : \Omega _0^{\odd} (X\setminus \Sigma  )  \longrightarrow  \Omega _0^{\ev} (X \setminus \Sigma  ). \end{equation}
The operator  $ D_t := D_{t}^{\ev}  + D_{t}^{\odd} $  is a closable operator with \begin{equation} \dom (D_{t,  \min }) =  \left \{ \begin{array}{l} \omega \in \LLL^2 \mid \text{ there exists a sequence } \phi _n \in \Omega ^*_0 (X) \text{ s.t. } \phi _n \ra \omega  \\
\text{ and } d _t \phi _n, \delta_t \phi _n  \text{ are Cauchy sequences in } \LLL^2 (\Lambda ^ *(T^*(X \setminus \Sigma))) \end{array} \right \} . \end{equation}

\noindent Thus in particular \begin{equation} \dom (D_{t,  \min }) \subset \dom (d _{t,\min}) \cap \dom (\delta _{t, \min}). \end{equation}

\begin{lemma}\label{vorbereitungstokes}
Let $X$ be a conformally conic Riemannian manifold of even dimension $\dim X = 2 \nu$.
\begin{itemize}
\item[(a)]  Then we have
\begin{equation} \dom (d_{t,k, \max} ) \cap \dom (\delta _{t, k-1 , \max} ) \subset \dom (D_{t, \min }) \end{equation} for all $k$, except possibly $k = \nu$.

\item[(b)] Moreover, for $k \not = \nu$, \begin{equation}
\begin{split}
& d_{t,k, \max}  = d _{t,k, \min }, \\
& \delta _{t,k-1, \max}  = \delta _{t,k-1, \min} ,\\
& \Delta _{t}^{ \FF, (k) }  = d_{t,k-1, \min } \delta _{t,k-1, \min } + \delta _{t,k, \min } d _{t, k, \min } ,\\ 
\end{split}
\end{equation} \noindent where $\Delta _{t}^{\FF}$ denotes the Friedrichs extension of  $\Delta _{t} \restriction_{ \Omega ^*_0 (X \setminus \Sigma)}$.
\end{itemize}
\end{lemma}

\begin{proof} (a) It is easy to see using the local form of an admissible function $f$ near the singularities and the formulas in Section \ref{sectionmodel} that $df \wedge : \LLL^2 \ra \LLL^2$ and $\nabla f \haken : \LLL^2 \ra \LLL^2$ are bounded operators. Therefore we get \begin{equation}\label{domainsdeformed}
\begin{split}
& \dom (d_{t,k, \max} ) = \dom (d_{k, \max} ) ,\\
& \dom (\delta _{t, k-1 , \max} )  = \dom (\delta _{ k-1 , \max} ) ,\\
& \dom (D_{t, \min })  = \dom (D_{ \min }),
\end{split}
\end{equation}
where $D = d+ \delta $.
By Theorem 2.1 in \cite{leschcone} we have, for $k \not =\nu$,
\begin{equation}\label{domainsderham} \dom (d_{k, \max} ) \cap \dom (\delta _{k-1 , \max} ) \subset \dom (D_{ \min }) . \end{equation}
\noindent The claim now follows from \eqref{domainsdeformed} and \eqref{domainsderham}.

(b) By (a) the assumptions of Lemma 3.3 in \cite{leschcone} are satisfied for the complex $(\Omega ^*_0( X \setminus  \Sigma)  , d_t, \langle \ , \  \rangle)$. By applying the cited result we get the claim. \end{proof}

\begin{prop}\label{propstokes}
\begin{itemize}
\item[(a)] The complex $(\Omega _0^*(X \setminus  \Sigma), d_t, \langle \ , \  \rangle)$ satisfies the $\LLL^2$-Stokes theorem, {\rm i.e.} \begin{equation} \dom (d_{t,\max} )= \dom (d _{t,\min}) .\end{equation} We denote by $(\CCC_t, d_t, \langle \ , \  \rangle)$ the unique extension of the  complex $(\Omega _0^*(X \setminus  \Sigma), d_t, \langle \ , \  \rangle)$ into a Hilbert complex.
\item[(b)] The  complex $({\CCC _t}, {d_t}, \langle \ , \  \rangle) $ is a Fredholm complex whose cohomology is isomorphic to $H^i_{(2)}(X)$, the $\LLL^2$-cohomology of $X$.  Moreover, for $ i = 0, \ldots, 2 \nu$,  the natural maps 
\begin{equation} \ker \Delta _{t}^{(i)} =  \ker { d_{t, i}}\cap \ker {\delta_{t,i-1}} \longrightarrow  H^i ({\CCC _t}, {d_t}, \langle \ , \  \rangle) \simeq H^i_{(2)}(X)   \end{equation} are isomorphisms, where $\Delta_t$ is the  Laplacian associated to the complex $(\CCC_t, d_t, \langle \ , \  \rangle)$. 
\item[(c)] The operator $\Delta _t$ is a  discrete operator and moreover, for  $k \not = \nu$,
\begin{equation} \Delta _t^{(k)} = \Delta _{t}^{ \FF , (k)}. \end{equation}

\item[(d)] The Gauss-Bonnet operator associated to the complex $(\CCC_t, d_t, \langle \ , \  \rangle)$, {\it i.e.} \begin{equation}\label{gaussbonnetdef} D_{t}^{GB} := \displaystyle  \bigoplus _{k \geq 0} ( d_{t,2k, \min} + \delta _{t,2k-1, \max}) \end{equation} is Fredholm and satisfies
\begin{equation}\label{gaussbonnet} D_{t}^{GB} :=
 \left \{ \begin{array}{ll} D_{t, \min}^{\ev} & \text{ if } \nu \text{ odd}, \\ D_{t, \max}^{\ev} & \text{ if } \nu \text{ even}.  \end{array} \right .\end{equation}
Moreover $\mathrm{ind} ( D_t ^{GB} ) = \chi _{(2)} (X)$ where $\chi _{(2)} (X)$ is the $\LLL^2$-Euler characteristic of $X$.

\end{itemize}
\end{prop}

\begin{remark} {\rm We call the  Laplacian $\Delta _t$ associated to the Hilbert complex $(\CCC _t, d_t, \langle\ , \ \rangle )$  the Witten Laplacian. By definition (of the Laplacian associated to a Hilbert complex) it is   the closed self-adjoint non-negative extension of  $\Delta _{t} \restriction _{ \Omega _0^* ( X \setminus \Sigma)}$ with domain:
 \begin{equation}\dom (\Delta _t ) = \big \{ \omega  \  | \  \omega,  d_t \omega,  \delta_t \omega, d_t \delta _t \omega, \delta _t d_t \omega \in \LLL ^2 \big ( \Lambda ^*(T^*(X \setminus  \Sigma)) \big ) \big \} . 
\end{equation} 

}
\end{remark}

\begin{proof} a) Let us denote by $ \langle \ , \  \rangle_t$ the twisted $\LLL^2$-metric: \begin{equation} \langle \alpha, \beta \rangle _t = \int_{X \setminus \Sigma} \alpha \wedge * \beta e^{-2tf}. \end{equation} Since $f$ is bounded on $X$ the two metrics $\langle \ , \  \rangle$ and $ \langle \ , \  \rangle _t$ are equivalent. We introduce the following auxiliary differential complex \begin{equation}\label{auxiliar}  (\Omega _0^*(X \setminus  \Sigma ),\widetilde{d_t}, \langle \ , \  \rangle_t),\end{equation} where $ \langle \ , \  \rangle_t $ is the twisted metric defined above  and $\widetilde{d_t}:= d$ is the usual differential. Using the invariance of the validity of the $\LLL^2$-Stokes theorem for equivalent metrics one deduces easily that the complex $(\Omega _0^*(X \setminus  \Sigma ),d, \langle \ , \  \rangle_t)$ admits a unique extension to a Hilbert complex, which we denote by
\begin{equation*} (\widetilde{\CCC _t}, \widetilde{d_t}, \langle \ , \  \rangle_t)= (\widetilde{\CCC _t}_{\max}, \widetilde{d_t}_{\max }, \langle \ , \  \rangle _t) = (\widetilde{\CCC _t}_{\min}, \widetilde{d_t}_{ \min}, \langle \ , \  \rangle_t). \end{equation*}

Since $d_t(e^{-tf} \omega) = e^{-tf}(\widetilde{d_t} \omega)$ the map
\begin{equation}\label{mapeft}
e^{-tf} : (\Omega _0^* (X \setminus  \Sigma ), \widetilde{d_t} , \langle \ , \  \rangle_t) \ra ( \Omega _0^* (X \setminus  \Sigma ) , d_t, \langle \ , \  \rangle), \quad \omega \mapsto e^{-tf} \omega
\end{equation}

\noindent is  an isomorphism of complexes. It is not difficult to see that the map \eqref{mapeft} extends to   isomorphisms of Hilbert complexes \begin{equation} (e^{-tf})_{\max/\min} : (\widetilde{\CCC _t}_{\max/\min}, \widetilde{d_t}_{\max / \min}, \langle \ , \  \rangle _t) \simeq  ({\CCC _t}_{\max/\min}, {d_t}_{\max/ \min}, \langle \ , \  \rangle).\end{equation} The claim  now follows from the validity of the  $\LLL^2$-Stokes theorem for the complex $(\widetilde{\CCC _{t}}, \widetilde{d_{t}}, \langle \ , \  \rangle_t) $.

 Alternatively we can prove the claim in a) also by the following argument: By Lemma \ref{vorbereitungstokes} (b)   one has 
\begin{equation} d_{t,k,\max} = d _{t,k, \min} , \ \delta _{t, k-1, \max} = \delta _{t, k-1, \min} \text{ for all  } k \not = \nu . \end{equation}
Therefore, since $\delta _{t,k,\min/\max}$ is the adjoint of $d_{t, k, \max/\min}$, we get  $d_{t,k,\max} = d _{t,k, \min}$ for all $k$.

(b) Since  $(\CCC, d, \langle \ , \  \rangle)$ is Fredholm, also $(\widetilde{\CCC _t} , \widetilde{d_t}, \langle \ , \  \rangle_t)$ is Fredholm. Thus by the isomorphism constructed in (a) also $(\CCC _t , d_t, \langle \ , \  \rangle )$ is Fredholm. The rest of the claim follows using the isomorphism in (a) and  general statements for Hilbert complexes in \cite{hilbert} (Theorem 2.4 and Corollary 2.5).

(c) We know already that $\Delta $ is discrete.  Since the  discreteness of the Laplacian associated to a Hilbert complex is invariant under complex isomorphism (see Lemma 2.17 in \cite{hilbert}), also $\Delta _t$ is discrete. The second claim follows from Lemma \ref{vorbereitungstokes}.

(d)  As for the complex of $\LLL^2$-forms  (compare \cite{leschcone}, Theorem 3.7 (c)) the closed extensions of $d_t + \delta _t$ are restricted by the relation \begin{equation}\label{restriction1} \dom ((d_t + \delta _t) _{\min} )  \cap \LLL^2 (\Lambda ^kT ^* X) = \dom (d_{t,k, \min}) \cap  \dom (\delta _{t,k-1,\min} ) , \quad k \not = \nu \end{equation} 
and \begin{equation}\label{restriction2} p_k \dom ((d_t + \delta _t )_{\max}) = \dom (d_{t,k, \min}) \cap \dom (\delta _{t, k-1, \min}) , \quad  |k - \nu| \not = 1  ,\end{equation}
where $p_k $ denotes the orthogonal projection $p_k : \LLL^2 (\Lambda ^* T^* (X \setminus \Sigma)) \rightarrow  \LLL^2 (\Lambda ^k T^* (X \setminus \Sigma))$.  Thus \eqref{gaussbonnet} follows from part (a), \eqref{gaussbonnetdef}, \eqref{restriction1} and \eqref{restriction2}. The Fredholm property of $D _{t}^{GB}$ is equivalent to the Fredholm property of the complex $(\CCC _t, d_t, \langle \ , \  \rangle)$ (see Theorem 2.4. in \cite{hilbert}). The last statement in (d) is now obvious. \end{proof}

Let $\Delta ^ {\FF} $ denote the Friedrichs extension of $\Delta \restriction _{ \Omega _0^*(X \setminus  \Sigma )}$. As a corollary we get

\begin{cor}\label{corstokes} \begin{itemize} \item[(a)] The form domains of $\Delta ^ {\FF} $ and $\Delta_t^ {\FF}$ coincide.  \item[(b)] For $k \not = \nu$ the form domain of  $  \Delta ^{(k)} $ and the form domain of the Witten Laplacian $\Delta ^{(k)} _t $ coincide. \end{itemize}\end{cor}

\begin{proof} (a) The form domain of $\Delta ^ {\FF}$ is the closure of $\Omega _0^*(X \setminus  \Sigma)$ under the norm \begin{equation} \norm{\omega }_1 ^2 := \norm{d \omega } ^2 + \norm{\delta \omega }^ 2 + \norm{\omega }^ 2.\end{equation} The form domain of $\Delta _{t}^ {\FF} $ is defined similarly. Note moreover that for $\omega \in \Omega _0^*(X \setminus \Sigma)$
\begin{equation}\label{normequ} \begin{split}   \langle \Delta \omega  , \omega  \rangle & = \norm{d \omega }^2 + \norm{\delta \omega }^2 \\ 
& \leq  2 \big ( \norm{d_t\omega }^2  +  \norm{\delta _t  \omega }^2 +t ^2 \langle \vert \nabla f \vert ^2 \omega ,  \omega \rangle   \big ) \\
& =   2 \big ( \langle \Delta _t \omega  , \omega  \rangle  +t ^2 \langle \vert \nabla f \vert ^2 \omega ,  \omega \rangle   \big ) .\end{split}\end{equation} Similarly  \begin{equation}\label{normequ2} \begin{split}   \langle \Delta _t\omega  , \omega  \rangle &  \leq  2  \big ( \langle \Delta  \omega  , \omega  \rangle  +t ^2 \langle \vert \nabla f \vert ^2 \omega ,  \omega \rangle   \big ) . \end{split}\end{equation}
 The claim  in (a) now follows easily since $\vert \nabla f \vert ^2$ is bounded on $X$. The claim in (b) is a consequence of Part a),  Proposition \ref{propstokes} (c) and the fact that $\dom (\Delta ^{(k)}) = \dom (\Delta ^{ \FF , (k)})$, $k \not =\nu $ (see Theorem 3.7 (b) in \cite{leschcone}). 

\end{proof}

\section{The local model} \label{sectionlocalmodel}

From now on we will assume that $f: X \rightarrow \RRR$ is an admissible Morse function, as in Definition \ref{defadmissible} (b).

In this section we  define a local model operator for the  Witten Laplacian near a singular point  $p \in \Sigma$. Let us recall  that a sufficiently small neighbourhood $U_p$ of $p$ is homeomorphic to $\cone(L_p)$, where $L_p$ is the link of the singularity.  In the following we denote by \begin{equation} \cone(L_p) := [0, \infty) \times L_p / \{0 \} \times L_p \end{equation}  the infinite cone over $L_p$. For $\epsilon > 0$ we denote by $\cone _{\epsilon }(L)$ the truncated open cone \begin{equation} \cone _{\epsilon }(L) := \{ (r, \varphi) \in \cone (L) \mid r < \epsilon \} .\end{equation}  

\subsection{The local Morse data}

 Let $B_{\epsilon}(p)$ be the closed $\epsilon$-ball around $p \in \Sigma$ in $X$. Let us choose $0 < \delta << \epsilon$. Since $p$ is an isolated singular point of $X$ the local Morse data of $f$  at $p \in \Sigma$ reduces in this case to the normal local Morse data.  They are defined as the pair of spaces
\begin{equation} (M_p, l_p^-) := \Big ( B_{\epsilon}(p) \cap f^{-1} ([f(p) - \delta, f(p) + \delta ]) , B_{\epsilon}(p) \cap f^{-1} ( f(p) - \delta ) \Big ) .\end{equation} The local Morse data are independent of the  choice of $0 < \delta << \epsilon$ small enough.   The set  
\begin{equation}  l^- _p:= B_{\epsilon}(p) \cap f^{-1} ( f(p) - \delta ) \end{equation} can be seen as an ``exit set'' for the negative gradient flow.   As in stratified Morse theory (see \cite{goresky}, pg. 66) we will call $l_p^-$ the lower halflink of $f$ at $p$. It is not difficult to see that the   pair $(M_p, l_p^-)$ is homeomorphic to  $(B_{\epsilon } (p) , l^-_p)$   and thus \begin{equation} IH^* (M_p, l_p^-) \simeq IH^* (\cone (L _p), l^-_p) , \end{equation} 
where $IH^*$ denotes intersection cohomology with closed support. 

For simplicity we will  write from now on $L$, $l^-$, etc. instead of $L_p$, $l_p^ -$, etc. in this section. We will also assume  that $f (p) =0$. 

\subsection{Definition of the  model operator. Local spectral gap theorem}\label{deflocalmodel}

Let us fix $\epsilon > 0$. Let $\eta _{\epsilon} : \cone(L) \rightarrow [0,1] $ be a smooth cutoff function, with $\eta_ {\epsilon}  \restriction _{  \cone _{\epsilon}(L) } \equiv 1  $  and $\eta _{ \epsilon  } \restriction _{  \cone (L) \setminus \cone _{ 2 \epsilon}(L)  } \equiv 0 . $ 

We denote by $(\cone(L), g_{\conf})$ the infinite cone  over $L$ equipped with the metric
\begin{equation}\label{metric}
g_{\conf} = \eta _{\epsilon} g + (1- \eta _{\epsilon}) g_{\cone}, \text{ where }g_{\cone} = dr^2 + r^2 g_L(0).
\end{equation}

 Let $f : \cone(L) \rightarrow \RRR$ be a function on the infinite cone such that:
\begin{equation}\label{function}f = \eta _{\epsilon} (f_1 + f_2 ) + (1- \eta _{\epsilon} ) f_1, \end{equation}
\noindent  where $f_1 = rh, f_2 = O(r^{1+ \delta})$, $\vert \nabla f \vert ^2 \geq a^ 2 > 0$.

Let $\langle \ , \ \rangle := \langle \ , \ \rangle_{\conf}$ be the $\LLL^2$-metric on forms induced by $g_{\conf}$. Let $ \big ( \Omega ^*_0 (\cone(L)), d , \langle \ , \ \rangle \big )$ be the de Rham complex of smooth compactly supported forms on the infinite cone $ \big ( \cone(L) , g_{\conf} \big )$. We denote by $(\Omega ^ *_0 (\cone(L)), d_t, \langle \  , \   \rangle)$  the complex obtained by deforming the de Rham complex by means of the function $f$, i.e. $d_t :=  e^{-tf} d e^{tf}$.

\newpage

\begin{theorem}\label{hodgelocal}
 \item[(a)] There is a unique Hilbert complex $(\DD _t, d_t, \langle \ , \ \rangle)$ extending the complex  \begin{equation*}   \big ( \Omega ^ *_0 (\cone(L)), d_t, \langle \ , \  \rangle \big  ).\end{equation*}  
\item[(b)] The complex $(\DD_ t,d _t,\langle \ , \  \rangle) $ is Fredholm.  The natural maps 
\begin{equation}\label{hodgeisomlocal} \ker (\DDD _{t}^{(i)} ) \ra H^i(\DD _t, d_t, \langle \ , \ \rangle) , \quad i = 0, \ldots , 2 \nu ,\end{equation} 
are isomorphisms, where  $\DDD_t$ denotes the Laplacian associated to the Hilbert complex $(\DD _t, d_t, \langle \ , \ \rangle)$ and $\DDD_t^{(i)}$ its restriction to $i$-forms. 
\item[(c)] The model Witten Laplacian $\DDD_t$ satisfies a local spectral gap theorem: there exists $c>0$ such that for $t$ large enough \begin{equation} \spec (\DDD  _{t}^{(i)}) \subset \{0 \} \cup  [ct^2 , \infty). \end{equation} Moreover all forms in $\ker ( \DDD _{t})$, as well as their derivatives have  exponential decay outside a small neighbourhood of the singularity.
\item[(d)] For the cohomology of the  complex $(\DD _t, d_t, \langle \ , \ \rangle)$ one gets \begin{equation} H^i(\DD _t, d_t, \langle \ , \ \rangle) \simeq IH ^i(\cone  (L), l^-). \end{equation}

\end{theorem}

\begin{remark}{\rm We call the operator $\DDD _{t}$ defined in Theorem \eqref{hodgelocal} (b) the model Witten Laplacian. By definition (of the Laplacian associated to a Hilbert complex), the Laplacian $\DDD_t$ associated to the  complex $ \big ( \DD_t,d _t, \langle \ , \  \rangle \big ) $ has  domain 
\begin{equation}\label{domainlocalwitten}
\dom (\DDD _{t} )=  \big \{ \omega \mid \omega , d_t \omega, \delta _t \omega, d_t \delta _t \omega,   \delta _t d_t\omega \in \LLL^2 (\Lambda ^*(\cone(L))) \big  \}.\end{equation} }\end{remark}

For simplicity of presentation we prove Theorem \ref{hodgelocal} only in the case of the infinite cone $\left ( \cone (L), g_{\cone} = dr^2 + r^2 g_L(0) \right )$ equipped with the conic metric and the function $f= rh$. The perturbation argument needed to extend the proof to the case of a conformally conic metric and the more general function $f= rh + O(r^{1+\delta})$  are detailed in \cite{ulcurve} for the case of a complex curve and can easily be adapted to the situation here.

\subsection{A useful unitary transformation}\label{sectionmodel}

From now on in this section we always treat the model case $\big ( \cone (L), g_{\cone} = dr^2 + r^2 g_L(0) \big )$, $f=rh$.

Let us denote by $n:= \dim L = 2 \nu -1$.  We denote by $\pi $  the projection $\pi : L \times \RRR^+ \ra L$ and following \cite{brueningseeley} we define bijective maps \begin{equation} \begin{array}{lccl} U_k : &  C_0^{\infty} (\RRR^+, \Omega ^ {k-1} (L) \oplus \Omega ^ {k} (L) ) &  \longrightarrow &  \Omega _0^k ( \cone (L) ) , \\
&  ( \phi _{k-1}, \phi _k) & \mapsto & r^{k-1-n/2} \pi ^ * \phi _{k-1} \wedge dr + r^{k-n/2} \pi^ * \phi _k, \\
\end{array} \end{equation}
which extend to unitary maps 
\begin{equation} U _k : \LLL^2   \big ( \RRR^+, \LLL^2 (\Lambda ^{k-1} T^*L\oplus \Lambda ^k T^* L, g_L(0))\big  ) \ra \LLL^2 \big ( \Lambda ^k T^*(\cone (L))  \big ),  \end{equation}

\noindent $ k = 0, \cdots, n+1$.

Let  $\Omega ^{\ev},\Omega ^{\odd}$ denote even and odd forms, respectively. The $U_k$'s induce unitary maps
\begin{equation*}\begin{array}{lllcl}  U_{\ev} & : &  C_0^{\infty} (\RRR_+ , \Omega ^*(L)) & \longrightarrow & \Omega _0^{\ev} (\cone(L)) , \\ &  &  (\phi _0, \ldots, \phi _n ) &\mapsto & (U_0(0, \phi _0), U _2(\phi _1, \phi _2), \ldots, U_{n+1} (\phi _n, 0)), \\
\\
 U_{\odd} & : &  C_0^{\infty} (\RRR_+ , \Omega ^*(L)) & \longrightarrow & \Omega _0^{\odd} (\cone(L)) , \\ & &  (\phi _0, \ldots, \phi _n ) & \mapsto & (U_1(\phi _0, \phi _1), U _3(\phi _2, \phi _3), \ldots, U_{n} (\phi _{n-1}, \phi    _n)).\end{array}  \end{equation*} 

\noindent  We denote by  $\D ^{\ev}$ the following operator acting on even forms 
\begin{equation} \D ^{\ev} := d + \delta : \Omega ^{\ev}_0 (\cone(L))\rightarrow \Omega ^{\odd} _0(\cone(L)) .\end{equation} In the following operators on the link are  labeled with  $\widetilde{\ }$.  An easy computation (see \cite{brueningseeley}, Section 5) shows that 
\begin{equation} U^{-1}_{\odd} \D^{\ev} U_{\ev} = \partial _r +r^{-1} S_0 ,\end{equation}
where $S_0$ is the operator \begin{equation} S_0:=  \left (  \begin{array}{cccccc}
c_0 & \tilde \delta & & & &\\ 
\tilde d & c_1 & \tilde \delta & & & \\
& \tilde d & c_2 & \tilde \delta & & \\
 &  & \ddots & \ddots & \ddots &  \\
& & & \tilde d & c_{n-1} & \tilde \delta \\
& & & &       \tilde d & c_n \\
\end{array} \right ),\quad  c_i := (-1)^i \left (i - \frac{n}{2}\right ). \end{equation}

Similarly for $\D ^{\odd} := d + \delta : \Omega ^{\odd}_0 (\cone(L))\rightarrow \Omega ^{\ev} _0(\cone(L))$ we get \begin{equation} U^{-1}_{\ev} \D^{\odd} U_{\odd} = - \partial _r +r^{-1} S_0 .\end{equation}

For the Laplacian $\DDD^{\ev/\odd} $ acting on even resp. odd forms on the infinite cone we then get: \begin{equation}\label{deft}\begin{split} & \TTT ^{\ev} := U^{-1}_{\ev} \DDD ^{\ev}  U_{\ev}= - \partial _r^2 + r^{-2} (S_0^2 + S_0) ,  \\ & \TTT ^{\odd}:= U^{-1}_{\odd} \DDD ^ {\odd} U_{\odd}= - \partial _r^2 + r^{-2} (S_0^2 - S_0) .   \end{split} \end{equation}

For the   operators \begin{equation}\begin{split} &  \D_{t}^{\ev} := d_t + \delta_t : \Omega ^{\ev} _0(\cone(L))\rightarrow \Omega ^{\odd}_0(\cone(L)), \\ 
& \D_{t}^{\odd} : = d_t + \delta_t : \Omega ^{\odd} _0(\cone(L))\rightarrow \Omega ^{\ev}_0(\cone(L))  \\ \end{split}\end{equation}

\noindent associated to the  deformed complex one computes easily that \begin{equation}\begin{split} &  U^{-1}_{\odd} \D_{t}^{\ev} U_{\ev} = \partial _r +r^{-1} S_0 + t T, \\ & U^{-1}_{\ev} \D_{t}^{\odd} U_{\odd} = -\partial _r +r^{-1} S_0 + t T \\ \end{split} \end{equation}

\noindent where $T$ is the operator \begin{equation} T =  \left (  \begin{array}{cccccc}
(-1)^0 h & \widetilde \nabla h \haken & & & &\\ 
\widetilde dh  & (-1) h  & \widetilde \nabla h\haken  & & & \\
& \widetilde dh & (-1)^2h & \widetilde \nabla h \haken & & \\
 &  & \ddots & \ddots & \ddots &  \\
& & & \widetilde dh  & (-1)^{n-1}  h   & \widetilde \nabla h \haken \\
& & & &      \widetilde  d h & (-1)^n h  \\
\end{array} \right ).\end{equation}


 For the model Witten Laplacian acting on even/odd  forms one gets:
\begin{equation}\label{tevoddt} \TTT _t^{\ev/\odd} := U^{-1}_{\ev/\odd} \DDD _t^{\ev/\odd}  U_{\ev/\odd}= \TTT ^{\ev/\odd} + t r^{-1} M_h + t^ 2T^2,  \end{equation} 
where \begin{equation}\label{Mh} \begin{split} M_h & := (TS_0 +S_0T ) \\ & \\ & =  \left ( \begin{array}{ccccc}
\widetilde{ \Lie _{\nabla h }} + \widetilde{ \Lie _{\nabla h}^*} + 2 (-1) ^ 0 c_0 h & & & &  \\
& & \ddots & & \\
& & & \widetilde{ \Lie _{\nabla h }} + \widetilde{ \Lie _{\nabla h}^* }+ 2 (-1) ^  {n-1} c_{n-1} h & \\
\end{array} \right ) \end{split}\end{equation}

\noindent and \begin{equation} T^2 = \vert \nabla f \vert  ^2 \cdot \mathrm{Id} = ( h^2 + \vert \widetilde \nabla h \vert ^2) \cdot \mathrm{Id}  .\end{equation}

In the sequel the following rescaling argument will be useful:

\begin{lemma}\label{lemrescaling} For $t >0$ denote by $R_t$ the  unitary rescaling operator \begin{equation} R_t (f(  r, \varphi ) ) = \sqrt{t} f(tr, \varphi). \end{equation} Then \begin{equation}  \DDD _t  = t^2 R_t \DDD _1R_t^{-1}. \end{equation}
\end{lemma}

\begin{proof} The claim follows by an easy computation. \end{proof}

\subsection{Proof of Theorem \ref{hodgelocal} (a)}

The proof is similar to that of Proposition \ref{propstokes} (a). \hfill $\Box$

\subsection{Proof of Theorem \ref{hodgelocal} (b)}

Recall that by Definition \ref{defadmissible} of an admissible Morse function there exists a constant $a>0$ such that \begin{equation}\label{fc} \vert \nabla f \vert ^2 = h^2 + \vert \widetilde \nabla h \vert ^2 \geq a^2 .\end{equation}
Moreover by \eqref{Mh}  $M_h$ is a bounded zeroth order operator on the link $L$, {\it i.e.} there exists $C>0$ such that  \begin{equation}\label{MC}   \| M_h \|  \leq C.\end{equation}

Let us fix $\delta >0$ arbitrarily small, $0 < \delta << \epsilon$.  Let $\eta _1 : \cone (L) \rightarrow [0,1]$ be a cutoff function with $\supp{\eta _1} \subset \cone_{\delta}(L)$ and depending only on the radial coordinate, $\eta _1( r , \varphi) = \eta _1(r)$.
Let us denote by 
\begin{equation}\label{taudef}\tau := \TTT +  2 r^{-1} (S_0 \tilde{T} + \tilde{T} S_0 ), \text{ where } \tilde{T} = \eta _1T. \end{equation} 
Let us denote by $\norm{\ }_{\tau}$ resp. $\norm{\ }_{\TTT}$ the two norms 
\begin{equation} \norm{u}_{\tau} ^2:= (\tau u , u) +\norm{u}^2,\end{equation}
\begin{equation} \norm{u}_{\TTT}^2 := (\TTT u,u) + \norm{u}^2 .\end{equation}

The operator $\pm \partial _r + r^{-1} (S_0 + r \tilde{T})$ will be treated as a perturbation of the operator  $\pm \partial _r + r^{-1} S_0 $. We apply the techniques in \cite{leschcurve}, Section 3 to handle the situation.

\begin{lemma}\label{lemma1}\begin{itemize} \item[(a)] The operator $\tau$ is a perturbation of $\TTT$ which can be written as \begin{equation} \tau  =  \TTT  + \sum _{i,j} \Phi _j^* C_{ij} \Phi _i, \end{equation} where $C_{ij}$ are operator valued functions with support in $\cone _{\delta}(L)$ and \begin{equation*} \max _{ij}\|C_{ij}\| =: \widetilde\epsilon (\delta) \end{equation*} can be made arbitrarily small by choosing  $\delta$ small enough. Moreover the operators $\Phi _i$ are controlled by $\TTT$, {\rm i.e.} there exists $c_i >0$ such that  \begin{equation} \norm{\Phi _iu}^2 \leq c_i \norm{u}_{\TTT} ^2\text{ for } u \in C _0 ^{\infty} ( \RRR _+ , \Omega ^{*-1} (L) \oplus \Omega ^{*} (L)). \end{equation}
\item[(b)]  There exists $c>0$ such that for $u\in C _0 ^{\infty} ( \RRR _+ , \Omega ^{*-1} (L) \oplus \Omega ^{*} (L))$ \begin{equation}\label{tau} (\tau u, u) \geq - c \widetilde \epsilon(\delta) \norm{u}^2 =: - \epsilon(\delta) \norm{u}^2. \end{equation}
\end{itemize}
\end{lemma} 

Proof: (a) A computation similar to that in \cite{leschcurve} shows the claim. (b) As a consequence of (a) one has 
\begin{equation} \vert \norm{u} _{\tau}^2 - \norm{u}_{\TTT}^2 \vert \leq c \widetilde \epsilon (\delta) \norm {u} _{\TTT}^2 \text{ for } u \in C _0 ^{\infty} ( \RRR _+ , \Omega ^{*-1} (L) \oplus \Omega ^{*} (L)) \end{equation}
for an appropriate constant $c>0$.
 Hence
\begin{equation} \langle \tau u,u \rangle  \geq (1- c\widetilde \epsilon (\delta)  ) \langle \TTT u, u\rangle - c \widetilde \epsilon (\delta) \norm{u}^2 \geq - c\widetilde \epsilon(\delta) \norm{u}^2, \end{equation}
since $\langle \TTT u,u \rangle \geq 0$ for $u\in C _0 ^{\infty} ( \RRR _+ , \Omega ^{*-1} (L) \oplus \Omega ^{*} (L))$. \hfill $\Box$

Using the boundedness of $M_h$ (see \eqref{MC}) there exist $\gamma >0$ such that for any cut-off function $\eta _2 : \cone (L) \rightarrow [0,1]$, $\supp (\eta _2) \subset \cone (L) \setminus \cone _{\gamma}(L) , \eta _{2} \restriction_{  \cone (L) \setminus \cone _{2 \gamma}(L)} \equiv 1$ we have
\begin{equation}\label{MC2} \vert \langle \eta _2 r^{-1} M_h u, u \rangle  \vert \leq \epsilon (\delta )  \norm{u}^2.\end{equation}
 We write  the term $r^{-1}M_h$ as a sum 
\begin{equation} r^{-1} M_h = V_1+ V_2 +V_3,\end{equation}
where $V_1 := r^{-1} (S_0 \tilde{T} + \tilde{T} S_0  )$ and $V_2 := \eta _2  r^{-1} M_h$.  Thus the complement $V_3$ is a compactly supported potential, whose support does not contain  the singularity.

 Recall that $\TTT_1$ is the operator defined in \eqref{deft} with $t=1$. We write the operator $\TTT_1$ as a sum of two operators, namely $\TTT_1 =: K +L $, where
\begin{equation}\label{defbc}\begin{split}  K:= \frac{1}{2} \tau   + V_2 + (\abs{\nabla f}^2 -a^2), \quad  L:= \frac{1}{2} \TTT + V_3 +a^2. \\ 
\\ \end{split} \end{equation}

Let us denote by $\TTT^{\FF}$ (resp. $\TTT_1^{\FF}$) the Friedrichs extension of  $\TTT \restriction _{ C _0 ^{\infty} ( \RRR _+ , \Omega ^{*-1} (L) \oplus \Omega ^{*} (L))}$ (resp. $\TTT_1  \restriction _{  C _0 ^{\infty} ( \RRR _+ , \Omega ^{*-1} (L) \oplus \Omega ^{*} (L))}$).

\begin{lemma}\label{lemma2} \begin{itemize} \item[(a)] We have $K \geq - \frac{3}{2} \epsilon (\delta)$ on $C _0 ^{\infty} ( \RRR _+ , \Omega ^{*-1} (L) \oplus \Omega ^{*} (L))$. 
\item[(b)]  The operator $L$ in \eqref{defbc} with $\dom (L) = \dom (\TTT ^  {\FF})$ is bounded from below and  $\spec_{\ess} (L )\subset \left[ a^2 , \infty \right ) .$ 
\item[(c)] We have $\spec_{\ess} \left( \TTT _1 ^{\FF} \right )\subset  \left[ \frac{a^2}{2} , \infty \right ) $. \end{itemize}\end{lemma}

Proof: (a)  The claim follows  using \eqref{tau}, \eqref{MC2} and \eqref{fc}. (b) The first claim is obvious from the definition of $L$, the fact that $V_3$ has compact support and $\TTT ^{\FF}\geq 0$.
We show next that, since $V_3$ is a continuous compactly supported potential, it is relatively compact with respect to $\left ( \frac{1}{2}\TTT ^{\FF} +a^2  \right ) $. We have to show that   every sequence $\phi _n$ such that $\norm{\left ( \frac{1}{2}\TTT ^{\FF} +a^2  \right ) \phi _n}^2 + \norm{\phi_n}^2 $  is bounded has a subsequence $\phi_{n_k}$ such that $V_3 \phi_{n_k}$ is convergent (in $\LLL^2$). From the boundedness of $\norm{\left ( \frac{1}{2}\TTT ^{\FF} +a^2  \right ) \phi _n}^2 + \norm{\phi_n}^2 $ one gets the boundedness of $\norm{\phi _n}_{H^1(\Omega)}$  for every bounded domain $\Omega$, $\supp (V_3) \subset \mathrm{int}(\Omega)$, $0 \not \in \Omega$. By Rellich's compactness theorem there exists a subsequence $\phi _{n_k}$ such that $\phi_{{n_k}}\restriction _{ \Omega}$ is convergent in $\LLL^2 (\Omega)$. This implies the $\LLL^2$-convergence of $V_3\phi_{{n_k}}$.


We can now apply Weyl's theorem (see e.g. \cite{reed}, XIII.4 Corollary 2) and get \begin{equation} \spec_{\ess} (L) = \spec_{\ess} \left ( \frac{1}{2} \TTT ^{\FF}  + a^2 \right ) \subset \left [a^2, \infty \right) . \end{equation} (c) Note first that, similarly to Corollary \ref{corstokes} one can show that  the form domains of $\TTT^ {\FF}$ and  $\TTT _1 ^{\FF}$ coincide. From (a) and (b) we deduce using the min-max principle (for the associated quadratic forms, see e.g. \cite{reed} Theorem XIII.2) that $\spec _{\ess} \left ( \TTT_1 ^{\FF} \right ) \subset \left[ a^2 - \frac{3}{2} \epsilon (\delta ) , \infty \right) $ and we get the claim by choosing $\epsilon (\delta) $ small enough. \hfill $\Box$

\

With the tools provided in Lemma \ref{lemma1} and Lemma \ref{lemma2} we can prove the local spectral gap theorem in all degrees $k \not = \nu$.  To prove the result also in middle degree $\nu$, we will use a ``complex argument''.  We formulate the argument in full generality in the language of Hilbert complexes introduced in \cite{hilbert}: Let $({\mathcal{ C }}, d, \langle \ , \ \rangle ) $ be a Hilbert complex. We denote by $\Dcal _i $ (resp. $\Rcal _i $) the domain (resp. the range) of $d_i$. By definition of a Hilbert complex the operators $d_i$ are closed and densely defined operators in some Hilbert space $H_i$ and moreover $\Rcal _i \subset \Dcal _{i+1}$.  Thus we have a complex
\begin{equation}\label{hilbertN} 0 \rightarrow \Dcal _0 \xrightarrow {d_0}  \ldots \xrightarrow  {d_{N-1}} \Dcal _{N} \ra 0. \end{equation}

\begin{lemma}\label{lemmahilbert}
Let $({\mathcal{ C }}, d, \langle \ , \ \rangle ) $ be a Hilbert complex as in \eqref{hilbertN} and let $k \in \NNN$, $0 <k<N $. Let us denote by $\Delta $ the Laplacian associated to the Hilbert complex and  by $({\mathcal{ C } ^*}, d^*, \langle \ , \ \rangle ) $ the dual Hilbert complex.  Assume that $ 0 \not \in \spec _{\ess} ( \Delta^{(i)})$ for all $i \not = k $.  Then
\begin{itemize} \item[(a)]  $\range d_i$ and  $\range d _i ^*$ are closed for all $i$.
\item[(b)] If  $0  \in \spec _{\ess} \left ( \Delta ^{(k)} \right )$, then it  is an eigenvalue  of infinite multiplicity. \end{itemize} 
\end{lemma}

\begin{proof} 

(a)   We denote by $\Dcal _i ^ *$ (resp. $\Rcal _i^ *$) the domain (resp. the range) of $d_i^ *$. We will first show by an induction that $\Rcal _i, \Rcal _i^*$ are closed for $0 \leq i < k$. Let us  define the operator $A_0$ by 
\begin{equation} A_0 := d_0 : H _0 \longrightarrow  H _1, \  u_0 \mapsto d_0 u_0. \end{equation} The operator $A_0$ is a closed operator with adjoint $A_0^* = d_0^ *$. Since by assumption $0 \not \in  \spec _{\ess}( A_0^* A_0) = \spec_{\ess} ( \Delta^ {(0)}) $ we deduce that $\range (A_0^*A_0) = \range ( A_0^* )= \Rcal _0^ *$ is closed.  Then by the closed range theorem also $\range (A_0 ) = \Rcal _0$ is closed.

 Let $0<i < k$ and let us assume that  $\Rcal _j, \Rcal _j ^*$ are closed for all $j <i $. We define the closed operator $A_i$ by
\begin{equation}  A_i  : \Dcal _i \cap \Dcal_{i-1} ^* \subset   H_i \longrightarrow H_{i-1} \oplus H_{i+1} , \quad   u  \mapsto (d_{i-1} ^* u,  d_i u).  \end{equation} By assumption $0 \not \in \spec _{\ess} (A_i^* A_i) =\spec _{\ess} (\Delta ^{(i)}) $. Therefore $\range (A_i ^* A_i) = \range (A_i^*) = \Rcal _{i-1} \oplus \Rcal _i^*$ is closed. Using the  induction hypothesis and the closed range theorem we deduce that $\Rcal _i $ and $\Rcal _i^ *$ are closed. 

Dually one can now consider the operator $B_{N-1}  := d_{N-1} ^ * : \Dcal ^*_{N-1}  \rightarrow H_{N-1}  $ and get that $\Rcal _{N-1}$ and $\Rcal _{N-1} ^ *$ are closed. By a downward  induction one can show that $\Rcal _i$, $\Rcal _{i}^ *$ are closed for $k \leq i \leq N-1$. 

(b) For any Hilbert complex the weak Hodge decomposition holds. Using in addition  Part (a) of the lemma we get that
\begin{equation} H_k = \HHH _k \oplus \range (d_{k-1} )\oplus \range ( d^ * _{k} ) , \end{equation}
where $\HHH _k := \ker \Delta ^{(k)} = \ker d _k \cap \ker d^* _{k-1}$. To prove the claim it is enough to show that  $\Delta^ {(k)}  \restriction _{ \range ( d_{k-1} ) \oplus \range ( d^ * _{k} ) } \geq c$ for some constant $c >0$. This can be seen as follows: Let $0 \not = \alpha \in \range ( d_{k-1} ).$ Then $\alpha = d _{k-1} \beta$ for some $\beta \in H_{k-1}$. Without loss of generality we can assume that $\beta$ is coclosed. Then by the Cauchy-Schwarz inequality
\begin{equation}\label{1} \langle \Delta  \alpha , \alpha  \rangle = \langle \Delta d \beta, d \beta \rangle = \norm{d^ * d \beta } ^2 = \norm{\Delta \beta }^ 2 \geq \frac { \langle \Delta  \beta,  \beta \rangle ^ 2}{ \norm {\beta}^2 } . \end{equation}

Since $0 \not \in \spec _{ess} (\Delta ^ {(k-1)}) $ and since $\beta$ is orthogonal to $\HHH _{k-1}$ there exists a constant $c$ such that 
\begin{equation}\label{2} \langle \Delta  \beta , \beta \rangle \geq c \norm{\beta}^ 2 .  \end{equation} From \eqref{1} and \eqref{2} we then get
\begin{equation}  \langle \Delta \alpha  , \alpha \rangle \geq c  \langle \Delta \beta  , \beta \rangle = c \norm{d \beta} ^ 2 = c \norm{\alpha}^ 2 .  \end{equation}

\noindent A similar argument works for $\alpha \in \range (  d^ *_k )$. \end{proof}

\

{\it Proof of Theorem \ref{hodgelocal} b):} Note that by Theorem 2.4 in  \cite{hilbert} the Fredholm property of the complex $(\DD_t, d_t, \langle \ , \ \rangle)$  is equivalent to $0 \not \in \spec_{\ess} (\DDD_t) $. Once the Fredholm property is proved, the isomorphism in \eqref{hodgeisomlocal} follows from standard arguments on the Hodge theory of Hilbert complexes (see \cite{hilbert}, Corollary 2.5). 

Let us show first that $0 \not \in \spec_{\mathrm{\ess}} (\DDD_t^{(i)}) $ for $i \not = \nu$:  By the rescaling argument in Lemma \ref{lemrescaling} it is enough to prove $0 \not \in \spec_{\ess}  ( \DDD_1^{(i)}  )$. Since $U^{-1}\DDD_1^{(i)}U = \TTT _1^{(i)}$, this is equivalent to  proving that $0 \not \in  \spec_{\ess}   ( \TTT _1^{(i)}  ) .$  But for $i \not = \nu$ we have  $\TTT _1^ {(i)} = \TTT _1^ { {(i)}, \FF }$ (this can be seen as in Lemma \ref{vorbereitungstokes}).   The claim follows using Lemma \ref{lemma2} (c). 

We consider now the case $i = \nu$: Let us assume that  $0 \in \spec_{\ess} \left ( \DDD _t^{ ( \nu ) } \right ) $. By the  first part of the proof we know that the assumptions of Lemma \ref{lemmahilbert} hold for the Hilbert complex $ \big ( \DD _t , d_t, \langle \ , \ \rangle  \big )$ (and $k = \nu$).  Thus we conclude that $0$ is an eigenvalue of $ \DDD _t^{ ( \nu ) }$ of infinite multiplicity.  Let $\{\phi _n \}_{n \in \NNN}$ be a sequence of orthogonal eigenforms in $\ker  \DDD _t^{(\nu) }$. Denote by $P^{\nu}_t$ the Dirichlet realisation of the Witten Laplacian on the truncated cone $\cone _1(L)$. (More precisely for $P^{\nu} _t$, we impose Dirichlet boundary conditions at $r = 1$ and the conditions in \eqref{domainlocalwitten} locally near the singularity.)  This is a non-negative operator with finitely many eigenvalues in each finite interval. Let $\chi : \cone (L) \rightarrow [0,1]$ be a cut-off function with $\supp \chi \subset  \cone _{1/2}(L)$. Using  Agmon type estimates similar to those in \cite{helffer} one can prove  that for $t $ large enough we get
\begin{equation} \langle P_t^ {\nu}  (\chi \phi _n), \chi \phi _n \rangle   = O(e^{-ct}) \norm{ \chi \phi _n}^2.\end{equation}
Using the  min-max principle  this gives a contradiction to the fact, that $P_t ^ {\nu}$ has only  finitely many eigenvalues in the interval say e.g.~$[0,1]$. Therefore we have shown that $0 \not \in \spec _{\ess} (\DDD_t ^ {( \nu ) })$ for sufficiently large $t$. By rescaling we get $0 \not \in \spec _{\ess} ( \DDD_t^ {( \nu ) } )$ for   all $t >0$.

\hfill $\Box$

\subsection{Proof of Theorem \ref{hodgelocal} (c)} 
Note that $\DDD_1$ is a non-negative operator. The proof of part (b) showed that $0 \not \in \spec_{ess} (\DDD_1)$.  Thus there exists a constant $c>0$ such that $\spec (\DDD_1) \subset \{0\} \cup [c, \infty)$. By the rescaling argument in Lemma \ref{lemrescaling} we deduce that $\spec (\DDD_t) \subset \{0\} \cup [ct^2, \infty)$.  The claim on the decay of the eigenfunctions follows from Agmon type estimates similar to those in \cite{helffer}, (see also \cite{ulcurve} for Agmon type estimates for a singular curve). \hfill $\Box$

\subsection{Proof of Theorem \ref{hodgelocal} (d)} 

To  prove  Theorem \ref{hodgelocal} (d) we adapt the {\it cone construction} proposed  in \cite{braverman}, Section 5  (resp. \cite{farber}, Section 3) in the context   of the Witten deformation on manifolds with boundary (resp. of the Witten deformation for polynomial differential forms on  non-compact manifolds). The proof follows closely the lines of proof in \cite{braverman}, \cite{farber}, one has to take care that the arguments go through at the cone point. We give the arguments of the proof in some detail for convenience of the reader.

Denote by $(\EEE^*, d, \langle \ , \ \rangle )$ the complex defined by
\begin{equation}\EEE^i := \{ \omega \in \Omega ^*(\cone (L) \setminus \{p\}) \mid \omega , d \omega \in \LLL^2 \text{ loc. at } p \}. \end{equation}

\noindent Recall that $\dim \cone (L) = 2 \nu$. It is not difficult to see (compare the  remark below) that the complex $(\EEE^*, d, \langle \ , \ \rangle)$ computes the $\LLL^2$-cohomology of the punctured cone, \begin{equation}\label{localcalc} H^*(\EEE^*, d, \langle \ , \ \rangle) \simeq H^*_{(2)} (\cone (L)) = \left \{ \begin{array}{ll} H^i _{(2)} (L)= H^i(L) &  i \leq \nu -1, \\ 0 & \text{ else}.\end{array} \right .\end{equation}

\noindent For cones the integration morphism gives a duality between  $\LLL^2$-cohomology and intersection homology with middle perversity. Thus 
 \begin{equation}\label{dual} H^*(\EEE^*, d, \langle \ , \ \rangle) \simeq IH^*(\cone (L)).\end{equation}

For $c >0$ let  $j _{c} :U^-_{c} \subset \cone (L) \setminus \{p\}$ be the open subset of $\cone (L)$ defined by
\begin{equation} U^-_{c} := \{ (r, \varphi ) \in \cone (L) \mid f(r, \varphi) < - c\}.\end{equation} Note that for $c' > c$ we have $U_{c'}^- \subset U_c ^-$. Moreover
\begin{equation}\label{defretract} \mathrm{int} (l^-) \subset U_{c'} ^- \subset U_c^-\end{equation}
are deformation retracts. We recall that $ l^-$ is the lower halflink defined in Section 3.2 (choosing $\delta > c' >c$ in the construction there) and  we denote by  $\mathrm{int } (l^-):= l^- \setminus \partial l^-$. 
We denote by $(\Omega ^* (U_{c}^-),d)$ the de Rham complex of smooth forms on $U^-_{c}$. Using the de Rham isomorphism between singular cohomology and de Rham cohomology as well as the homotopy invariance of singular cohomology we get
\begin{equation} H^*(\Omega ^*(U_{c}^-), d) = H^*_{dR} (U^-_{c}) \simeq H^* _{\sing} (U_{c}^-)\simeq H^*_{\sing}(l^-).\end{equation}

\noindent Since $l ^- $ is smooth we have moreover that $H^* _{\sing} (l^-) \simeq IH^*(l^-)$ and therefore 
\begin{equation}\label{dual2} H^*(\Omega ^*(U_{c}^-), d)  \simeq IH^*(l^-).\end{equation}

We denote by $j_c^*$ the restriction map
\begin{equation} j_c^* : \EEE ^i \rightarrow \Omega ^i (U^-_c) , \  \omega \mapsto \omega _{\mid U_c^-}. \end{equation} The cone complex $\Cone (j_c^*)$ is defined as follows (see e.g. \cite{dold}) \begin{equation} \Cone ^k (j_c^*) = \EEE ^k \oplus \Omega ^{k-1} (U^-_c), \ d _{\Cone} (\eta, \eta ') = (d \eta, - d \eta '  + j_c^* \eta) .\end{equation} By construction of the cone complex and in view of \eqref{dual} and \eqref{dual2} we get
\begin{equation}\label{cohomologyofcone}  H^* (\Cone (j^*_c)) \simeq IH^* (\cone (L), l^-).\end{equation}

\noindent Note that because of \eqref{defretract} the isomorphism \eqref{cohomologyofcone} holds for all $c >0$.

\begin{remark}{\rm Some remarks may be in order here:\begin{itemize} \item[(1)] While all forms in the complex  $(\DD_t, d_t, \langle \ , \ \rangle)$ are $\LLL^2$-integrable on the infinite cone, for the complex $(\EEE, d, \langle \ , \ \rangle )$ we allow forms which are not necessarily $\LLL^2$-integrable at $\infty$.
\item[(2)] In \cite{cheeger2} the complex of $\LLL^2$-forms on the truncated cone $\cone _1(L)$ is used to compute the $\LLL^2$-cohomology of the punctured cone. Here we use the complex  $(\EEE, d, \langle \ , \ \rangle )$ on the infinite cone instead. However to prove \eqref{localcalc} one can proceed as in \cite{cheeger2} and construct homotopy operators which contract to the tip point (resp. to the base $L$ of the cone) if the form degree is $i \geq  \nu $ (resp. $i < \nu $). 
\end{itemize}}
\end{remark}

Let $( \DD_t^{\infty}  , d_t, \langle \ , \ \rangle)$ be the subcomplex of $( \DD_t , d_t, \langle \ , \ \rangle)$ of smooth forms, {\it i.e.} $\DD _t^{\infty, i } := \DD _t^i \cap \Omega ^i (\cone (L) \setminus \{p\})$. By the usual regularisation argument \begin{equation} \label{regularisation} H^*  ( \DD_t^{\infty}  , d_t, \langle \ , \ \rangle) \simeq H^*( \DD_t , d_t, \langle \ , \ \rangle).\end{equation} Let us recall that a form $\omega \in \DD_t ^{\infty} $ can be decomposed as \begin{equation} \omega= \omega _{\perp} + dr \wedge \omega _{||},\end{equation} where $\omega _{\perp}$ and $ \omega _{||}$ are forms on the link depending smoothly on $r$.

\begin{deflem} We can define the following  chain map  \label{defchainmap}
\begin{equation}\label{chainmap} \begin{array}{cclcl} \Phi &: &( \DD_t^{\infty} , d_t, \langle \ , \ \rangle) & \longrightarrow & \left (\Cone (j_c^*), d_{\Cone} \right ) \\
& &\omega= \omega _{\perp} + dr \wedge \omega _{||} & \mapsto & (e^{tf}\omega ,  \omega') , \end{array}\end{equation}
where \begin{equation}\label{xidef} \omega ' (r, \varphi)= - \int _r^{\infty} e^{tf(\tau , \varphi)} \omega_{||} (\tau, \varphi ) d \tau \text{ for } (r,\varphi) \in U_c^-. \end{equation}
\end{deflem}
\begin{proof} Note first that since $f < -c$ on $ U_c^-$ and $\omega \in \DD_t^{\infty} $ the integral in \eqref{xidef}  converges. Moreover to check that $\Phi$ is a chain map one has to verify that, on $U_c^-$,
\begin{equation}\label{chain} - \int _r^{\infty} e^{tf} ( d_t \omega) _{||} d \tau =  d \left ( \int _r^{\infty} e^{tf} \omega _{||} d \tau  \right ) + e^{tf} \omega .\end{equation}

As in Section 3.3, let us denote by $\tilde{d}$ the operator on the link $L$. The equation \eqref{chain} follows from the following computation \begin{equation}\begin{split}  - \int _r^{\infty} e^{tf} ( d_t \omega) _{||} d \tau&  =  - \int _r^{\infty} e^{tf} ( d \omega + t d f \wedge \omega ) _{||} d \tau  \\ 
& = - \int _r^{\infty} e^{tf} \left (  \frac{\partial \omega _{\perp}}{\partial \tau } - \tilde{d}\omega _{||} - t \tau d h \wedge \omega _ {||} + t h \omega _{\perp} \right) d \tau \\
& = - \int _{r}^{\infty} \frac{\partial}{\partial \tau } (e^{tf} \omega _{\perp} ) d \tau  + \int _r^{\infty}  \tilde{d}\left ( e^{tf}  \omega _{||} \right ) d \tau  \\
&=  e^{tf} \omega _{\perp} + dr \wedge e^{tf} \omega _{||} + d \left ( \int _r^{\infty} e^{tf} \omega _{||} d \tau  \right ) . \\
  \end{split}\end{equation}

\end{proof}

In view of \eqref{cohomologyofcone} and \eqref{regularisation} in order to prove Theorem \ref{hodgelocal}(d) it is enough to prove the following proposition:

\begin{prop} The chain map $\Phi$ defined in \eqref{chainmap} is a quasi-isomorphism. \end{prop}

\begin{proof} To prove the proposition we have just to check that the arguments in  \cite{braverman}, \cite{farber} adapt to our singular situation.  We denote by $\Phi^*: H^*(\DD_t^{\infty} , d_t, \langle \ , \ \rangle ) \rightarrow H^* (\Cone (j_c^*), d_{\Cone})$ the induced map on cohomology level. We prove first the injectivity of $\Phi^*$: 

Let $[\omega] \in H^*(\DD_t^{\infty} , d_t, \langle \ , \ \rangle )$ with $\Phi ^*[\omega] =0$. In view of the Hodge theory for the complex  $(\DD_t , d_t, \langle \ , \ \rangle)$ (part (b) of Theorem \ref{hodgelocal})  and elliptic regularity we can  choose a representative $\omega \in \ker \DDD _t$ for the class $[\omega]$. The fact that $\Phi ^*[\omega] =0$ means that there exists $(\alpha , \alpha ') \in \EEE^{i-1} \oplus \Omega ^{i-2}(U_c^-) $ such that \begin{subequations}\begin{align} & d \alpha = e^{tf} \omega  , \label{alpha1}\\ & -d\alpha ' + j_c^* \alpha = \omega '. \label{alpha2} \end{align} \end{subequations}

To show injectivity of $\Phi^*$ we have to find $\zeta \in \DD_t^{\infty, i-1} = \dom (d_{t,i-1}) $ with $d_t \zeta  = \omega$. This will be done by modifying the form $\alpha$ as in \cite{braverman}. The main point is to notice that  the proof also works in our situation since all the modifications are done  outside a small neighbourhood of the cone tip. 

\underline{\it Step 1:} Let $c<c'$. Let $\chi _1: \cone (L) \rightarrow [0,1]$ be a cutoff function such that $\chi _ {1 |U_{c'}^-} = 1 $  and $ \chi _{1 \mid \cone(L) \setminus U_c^-}=0$. 

Set  \begin{equation}\label{defbeta} \beta := \alpha - d ( \chi _1 \alpha ') .\end{equation}

Then we get \begin{subequations}\begin{align} & d \beta = e^{tf} \omega  \text{ directly from  the definition of } \beta \text{ and } \eqref{alpha1} , \label{beta1} \\ & \beta _{\mid U_{c'}^-} = \omega ' \text{ from the definition of }  \beta \text{ and }  \eqref{alpha2}. \label{beta2}\end{align}\end{subequations}

\underline{\it Step 2:} Let $\chi _2: \cone (L) \rightarrow [0,1]$ be a cutoff function  such that $\chi _2 = 1 $ for $r\geq 2 $ and $\chi _ 2 =0$ for $r \leq 1$. 
Set \begin{equation}\label{gammadef} \gamma := \beta - d \left( \chi _2 \int _{2} ^r \beta _{||} d\tau \right ) .\end{equation}

Then we get by using   \eqref{beta1} and \eqref{gammadef} that \begin{subequations}\begin{align} & d \gamma = e^{tf} \omega , \label{gamma1} \\ 
& \gamma _{||} = 0 \text{ on } r \geq 2  . \label{gamma3}  \end{align}\end{subequations}

From \eqref{gamma1}  and \eqref{gamma3} we get in particular that  $d \gamma _{\perp} = e^{tf} \omega $ and therefore \begin{equation}\label{diffequation} \frac{\partial }{\partial r } \gamma _{\perp} = e^{tf} \omega _{||} \text{ for } r \geq 2 . \end{equation}

\underline{\it Step 3:} For $\epsilon \in \RRR$ set  $S_{\epsilon } := \left \{ \varphi \in L \mid h( \varphi ) < \frac {\epsilon}{t}  \right \}$.  From Proposition 3.1. (c) we know that all forms  in $\ker \DDD _t$ have exponential decay, thus there exists constants $a>0$, $C>0$ such that  \begin{equation} \label{agmon} \omega \leq Ce^{-ar} \end{equation} outside a small neighbourhood of the singularity. Hence \begin{equation}\label{decay15}  e^{tf} \omega \leq C e^{-ar + \epsilon r } \text{ if } (r,\varphi ) \in S_{\epsilon} \times [2, \infty). \end{equation}

Therefore the integral $\int _{r}^{\infty} e^{tf} \omega _{||} d \tau $ is well defined and the differential equation \eqref{diffequation} on $S_{a} \times [2, \infty)$ is solved by 
\begin{equation}\label{gammaperp} \gamma _{\perp} = \xi - \int _{r}^{\infty} e^{tf} \omega _{||}d \tau,
 \end{equation}
where $\xi$ is a differential form on $S_{a}$ (which does not depend on the radial coordinate $r$).

From \eqref{gamma1} and  \eqref{gammaperp} we get that 
\begin{equation}\label{etaclosed1} e^{tf} \omega _{\perp} = \tilde d  \gamma_{\perp}   = \tilde d  \xi -  \int _{r}^{\infty}  \tilde d (e^{tf} \omega _{||} ) d \tau ,   \text{ on }  S_{a} \times [2, \infty). \end{equation}

Now since $d_t \omega = 0$ we have \begin{equation}0 =  (d_t \omega)_{||}  = e^{-tf} \left ( -\tilde d (e^{tf} \omega _{||} ) + \frac{\partial}{\partial r} (e^{tf} \omega _ {\perp} )  \right )  , \end{equation}
and thus 
\begin{equation}\label{etaclosed2} -\tilde d (e^{tf} \omega _{||} ) + \frac{\partial}{\partial r} (e^{tf} \omega _ {\perp}) = 0. \end{equation}

Inserting \eqref{etaclosed2} into \eqref{etaclosed1} we get that (on $S_{a} \times [2, \infty)$) \begin{equation}\label{etaclosed3} e^{tf} \omega _{\perp} =  \tilde{d}\xi + e^{tf} \omega _{\perp} \end{equation}
and thus $\tilde{d}\xi = 0$, {\it i.e.} $\xi$ is a closed form on $S_{a}$.

An explicit computation using  \eqref{beta2}, \eqref{gammadef} and \eqref{gammaperp} shows that \begin{equation*} \xi = -d (\left ( \int _{2} ^{\infty} e^{tf} \beta _{||} \right ) \text{ on } U_{c'}^- .\end{equation*} Note that for $c'$ large enough $U_{c'} ^ - \subset \{ r \geq 2\}$ and $\xi$ does not depend on the radial coordinate $r$. Moreover for each $\varphi  \in S_0$ there exists $r_{\varphi}$ such that $(r_{\varphi}, \varphi ) \in U_{c'}^-$ we conclude that $\xi _{ \mid S_0} $ is exact. Thus $[\xi ] = 0 \in H^*(S_0)$. Since by definition of an admissible Morse function (see also \eqref{fc}) $0$ is a regular value of $h$ we can find $\epsilon \in (0, a/2)$ such that $H^*(S_0) \simeq H^*(S_{2 \epsilon})$. Thus $[\xi ] = 0 \in  H^*(S_{2\epsilon})$, {\it i.e.} there exists a form $\mu \in \Omega ^*(S_{2\epsilon})$ with $d \mu = \xi$. Let $\chi _3 : L \rightarrow [0,1]$ be a cut-off function with $\chi _{3  \mid S_{4/3 \epsilon}}  \equiv 1 $ and $\chi_{3  \mid L \setminus S_{5/3 \epsilon}}   \equiv 0$. Set 
\begin{equation}\label{thetadef} \theta := \gamma - d( \chi _3 \chi _2 \mu). \end{equation}

We have \begin{subequations}\begin{align} & d \theta = d \gamma = e^{tf} \omega \label{theta1}, \\
 & \theta _{\mid S_{\epsilon} \times [2, \infty)} =  \gamma _{\perp} - d \mu = - \int _{r} ^ {\infty} e^{tf} \omega _{||} d \tau \text{ by using } \eqref{gammaperp} ,\label{theta2} \\ & \theta _{||} = 0 \text{ on } r > 2 \text{ by the definition of } \theta \text{ and } \eqref{gamma3}. \label{theta3}
 \end{align}  \end{subequations}

Set $\zeta := e^{-tf} \theta$. Then $d_t \zeta = \omega$. The proof of the injectivity of $\Phi ^*$ is completed if we can show that $\zeta \in \DD _t ^{\infty} = \dom (d_t)$. Since $d_t \zeta = \omega \in \LLL ^2 $ we need to show only that   $\zeta \in \LLL^2 (\Lambda ^*(T^*(\cone (L)))$. Note first that since  the cut-off functions $\chi_1$, $\chi_2$ used to construct $\theta$ from $\alpha$ have support outside the cone point and $\alpha \in \LLL^2$ locally near the cone point we have $\zeta _{\mid \cone _{2 } (L)} \in \LLL^2$. One uses \eqref{decay15} and \eqref{theta2} to show that $\zeta_{\mid S_{\epsilon } \times [2, \infty )}$ has exponential decay for $r \ra \infty$ and thus $\zeta _{\mid S_{\epsilon } \times [2, \infty )} \in \LLL^2$.   On $ ( L \setminus S_{\epsilon } ) \times [2, \infty)$ we use \eqref{theta1} and \eqref{theta3} 
and thus
\begin{equation}\label{thetaperp} \frac{\partial}{\partial r} \theta _{\perp} = e^{tf} \omega _{||} .\end{equation}  

By integrating \eqref{thetaperp} one gets 
\begin{equation}\label{thetaperp2} \theta _{\perp} = \theta _{\perp } ( \varphi, 2) + \int _{2} ^ r e^{tf} \omega _{||} d \tau .\end{equation}

Again using \eqref{agmon} and \eqref{thetaperp2} one gets  that $\zeta _{\mid ( L \setminus S_{\epsilon } ) \times [2, \infty)} $ is exponentially decaying and thus $\zeta _{\mid ( L \setminus S_{\epsilon } ) \times [2, \infty)} \in \LLL^2$. This finishes the proof of the  injectivity of $\Phi^*$.

\ 

We now show the surjectivity of $\Phi^*$: We can choose $c >0$ large enough such that $U_c ^- \subset S_{\epsilon} \times [3, \infty).$ Moreover $U_c^{-} $ is a deformation retract of $S_{\epsilon} \times [3, \infty).$ As in \cite{braverman} one can show that a class $[(\alpha, \alpha _1)] \in H^i (\Cone (j_c^*)) $ can be represented by a pair $(\gamma , 0 ) \in \EEE ^i \oplus \Omega ^{i-1} (U_c^-)$ such that $d \gamma = 0$, $\gamma \restriction _{S_{\epsilon} \times [3, \infty)} = 0$ and $\gamma $ is bounded. We then have $e^{-tf} \gamma \in \LLL^2$: Since $\gamma \in \EEE ^i $ in particular $e^{-tf} \gamma \restriction_{\cone _3(L)} \in \LLL^2$. Moreover $\gamma \restriction _{S_{\epsilon} \times [3, \infty)} = 0$ and on $(L \setminus S_{\epsilon}) \times [3, \infty)$ we have $e^{-tf} < e^{-t \epsilon r} $. Since $\gamma $ is bounded we deduce that $e^ {-tf} \gamma \restriction _{(L \setminus S_{\epsilon}) \times [3, \infty)} \in \LLL^2$. We conclude that $e^{-tf } \gamma \in \DD^{\infty , i}  _t $ and moreover $\Phi ^*( [ e^{-tf} \gamma ] ) = [(\gamma, 0)]=[(\alpha, \alpha _1)]$.

\end{proof}

\begin{example}{\rm Let us illustrate the result in Theorem \ref{hodgelocal} for the case $f = \pm r.$  In this case we have separation of variables for the model Witten Laplacian $\DDD _t$ and the lower halflink is easy to determine. Therefore we can compute both sides of the isomorphism
\begin{equation}\label{isomthm} \ker (\DDD_t^{(i)}) \simeq IH^i(\cone (L), l^-)\end{equation} 
explicitely. In the following we denote by $\HHH ^i(L)$  the harmonic $i$-forms on the link $L$. 
\begin{itemize} \item[(a)] Let $f = r$. We get, with the notations as in Section \ref{sectionmodel} (for the first order operator acting on $i$-forms)

\begin{equation} U^{-1}\D _t U=  \pm \partial _r + r^{-1} \left ( \begin{array}{cc} c_{i-1} & \widetilde \delta \\ \widetilde d & c_{i} \end{array}  \right ) + t \left ( \begin{array}{cc}  (-1) ^{i-1} & 0 \\ 0 & (-1) ^{i} \end{array}  \right )\end{equation} and an explicit computation shows that 
 \begin{equation}\label{kernel+} \ker \DDD_t ^{(i)}= V^i _+\end{equation}
where 
 \begin{equation} V ^i _+:= \left \{ \begin{array}{ll} \mathrm{span} \{ e^{-tr} \eta \mid \eta \in \HHH ^i (L) \}  &  \text{ if } i < \nu ,\\
 0 & \text{ if } i \geq \nu .\end{array}  \right .  \end{equation} 
 Note that for $i < \nu$ the forms in $V^i_+$ are locally $\LLL^2$ near the cone point, the factor $e^{-tr}$ ensures that they are $\LLL^2$ at $\infty$. Recall that $\dim \cone (L)= 2 \nu$.

 On the other hand $l^-= \emptyset$ and from the local calculation for intersection homology with middle perversity (see \cite{goresky2}) we know that \begin{equation}\label{localcalc1} IH^i (\cone (L), \emptyset ) = IH^i (\cone (L)) = \left \{ \begin{array}{ll} H^i(L) & \text{ if } i <  \nu ,  \\ 0 & \text{ if } i \geq \nu .\end{array} \right . \end{equation}

 Since by Hodge theory for the smooth manifold $L$ we have $\HHH^i(L) \simeq H^i(L)$, the isomorphism \eqref{isomthm} follows from \eqref{kernel+} and \eqref{localcalc1}.

\item[(b)] $f= -r$. Again an explicit computation shows that \begin{equation} \label{kerminus}  \ker (\DDD _t^{(i)})  = V^i _- ,\end{equation} 
where
\begin{equation}  \label{vminus}  
V^{i} _-  : = \left \{ \begin{array}{ll} 0 & \text{ if } i \leq \nu , \\
\mathrm{span} \{ e^{-tr} r^{-n+2(i-1)} dr \wedge \eta \mid \eta \in \HHH ^{i-1} (L) \} &  \text{ if } i  > \nu .\\ \end{array}  \right .  
\end{equation}  Note again that the form $e^{-tr} r^{-n+2(i-1)} dr \wedge \eta$, $h \in H^ {i-1}(L)$, is $\LLL^2$-integrable if and only if $i > \nu$.

 On the other hand we have $l^- =  L$ and by the local calculation for the intersection homology with middle perversity
\begin{equation} \label{vminus2}  
IH^ i (\cone (L), L) =  
\left \{ \begin{array}{ll} 0 & \text{ if } i \leq \nu,  \\
H ^{i-1} (L) &  \text{ if } i  > \nu .\\
 \end{array}  \right .   \end{equation}

 Again \eqref{kerminus}  and \eqref{vminus2}   show the  isomorphism \eqref{isomthm} explicitely for this case.

Note that comparing (a) and (b) one recognises the usual  Poincar\'e duality in Morse theory, namely that the Hodge star operator $*$ yields an  isomorphism  \begin{equation}\label{poincare} * : V^i_- \simeq V^{2 \nu -i} _+. \end{equation}

\end{itemize}  }
\end{example}

\begin{example}{\rm
In \cite{ul} the case of a stratified Morse function in the sense of Goresky and MacPherson on a singular complex curve has been treated. In this case the local model for a (unibranched) singular point $p$ of multiplicity $m$ is given by $\left ( \cone (S_m^1), dr^2 + r^2 d \varphi ^2 \right )$, where $S_m^1$ denotes the circle of length $2 \pi m$. A stratified Morse function (in the coordinates $(r, \varphi)$) is just $f = r \cos \varphi$. The local model for the Witten Laplacian is then ``$\DDD_ t = \DDD + t^2$'' (but $\dom (\DDD _t) \not = \dom (\DDD)$!) and an explicit computation shows that 
\begin{equation} \dim \ker \DDD _t = \dim \ker \DDD _t ^{(1)} = m-1.  \end{equation}
The elements in $\ker  \DDD _t$ involve  the modified Bessel functions $K_{\nu}$, where  $\nu ^2\in (0,1)$ is an eigenvalue in  $\spec (\Delta _{S_m^1}) = \left \{ 0, \frac{1}{m^2}, \frac{4}{m^2} , \ldots \right \}$.
On the other hand here the lower halflink $l^-$ is homotopic to $m$ points and therefore 
\begin{equation} 
\dim IH^ i (\cone (S^1_m), l^-) =  
\left \{ \begin{array}{ll} 0 & \text{ if } i =0, 2,  \\
m-1 &  \text{ if } i  = 1 .\\
 \end{array}  \right .   \end{equation}
}

\end{example}

\section{Proof of the spectral gap theorem and Morse inequalities}\label{sectiongap}

{\it Proof of Theorem \ref{thmspectralgap}(1)}: The proof of the spectral gap theorem consists in two steps. The first step, namely the study of a model operator for the Witten Laplacian in the neighbourhood of a singular point $p \in \Sigma$ of $X$ has been done in Section 3. In the second step of the proof it is now enough to follow the strategy of proof in the smooth case. We follow here the proof in  \cite{bismutlebeau}, Section 9 and just give a very rough outline. (In \cite{ul} the proof has been detailed for the case of a complex curve with cone-like singularities). Recall from the smooth theory that the model Witten Laplacian $\DDD_{t,p}$ in the neighbourhood of a critical point $p \in X \setminus \Sigma$ of $f$ of index $i$ has discrete spectrum $\spec (\DDD _{t,p} ) = 2 \NNN t$ and $\dim \ker (\DDD _{t,p})= \dim \ker (\DDD _{t,p}^{(i)})  = 1$. We denote by $\omega _{p,1}^i(t)$ the generator of $\ker (\DDD _{t,p})$. For a singular point $p \in \Sigma $ we denote by $\{ \omega _{p,j}^i(t) \mid j = 1, \ldots,  m_p^i\}$ a bases of $\ker (\DDD _{t,p}^{(i)})$, $i = 0, \ldots , 2 \nu.$ Let $\nu _{\epsilon} : \RRR^ + \ra \RRR$ be a cut-off function with $\nu _{\epsilon }= 1$ in $[0,  \epsilon /4]$ and $\supp \nu_{\epsilon}  \subset [0, \epsilon /2]$ ($\epsilon$ as in Section \ref{sectionlocalmodel}). The forms $\Phi _{p,j}^i(t):= \nu_{\epsilon }(| x |  ) \omega ^i_{p,j}(t)$ can be identified with $\LLL^2$-forms on $X$.  For $i = 0, \ldots, 2 \nu$ we denote by 
\begin{equation*} E^i(t) := \spa \big \{   \{ \Phi _{p,1}^i(t) \mid p \in \Crit _i(f) \setminus \Sigma  \} \cup \{\Phi _{p,j}^ i(t)  \mid  p \in \Sigma ,  j \in I_p^i := \{ 1, \ldots , m_p^i \} \} \big \}. \end{equation*} We get  an orthogonal splitting $\LLL^2 (\Lambda ^*(T^*(X \setminus \Sigma))) = E(t) \oplus E(t)^{\perp}$. The closed  operator $D_t :=  d_t + \delta  _t$ with $\dom (D_t) = \dom (d _t) \cap \dom ( \delta  _t) \subset \LLL^2 (\Lambda ^*(T^*(X\setminus \Sigma))) $ can be written in matrix form

$$D_t = \left ( \begin{array}{cc} 
D_{t,1} & D_{t,2} \\
D_{t,3} & D_{t,4} \\
\end{array} \right ) \mbox { according to the splitting } E(t) \oplus E(t)^{\perp}. $$ \noindent Note that $\dom (D_t)$ equipped with the  norm  $\norm{ \omega }_1 := \sqrt{\norm{(d + \delta) \omega}^2 + \norm{\omega}^2 }$ is complete. We show  the following estimates on $D_t$ as $t \ra \infty$:

\begin{prop}\label{estimatea1}
There exist constants $c,C >0$ and $t_0 >0$ such that for all $t>t_0$ we have
\begin{enumerate}
\item For all $ \omega \in E(t) $ we have $\norm{D_t \omega} = O(e^{-ct} ) \norm{\omega} $. In particular $\norm{D_{t,1} \omega} = O(e^{-ct}) \norm{\omega}, \norm{D_{t,3} \omega} = O(e^{-ct}) \norm{\omega}.$ \item For all $\omega \in E(t)^{\perp} \cap \dom (D_t) $ we get: $\norm{D_{t,2}\omega} \leq O(e^{-ct}) \norm{\omega},  \norm{D_{t,4}\omega} \geq C(\norm{\omega}_1 + \sqrt{t}   \norm{\omega}).$\end{enumerate} \end{prop}

\noindent The proof of Proposition \ref{estimatea1} is similar to the corresponding statements in the smooth case (see \cite{bismutlebeau}, Section 9).  To prove the estimates  for forms $\omega$ with support in a neighbourhood of a singular point of $X$ the local spectral gap theorem and the decay of eigenforms of the model operator (see Theorem \ref{hodgelocal} (c)) are  crucial. As in \cite{bismutlebeau}, Section 9 (c) and (e), Proposition \ref{estimatea1} allows to give estimates for the resolvent of $D_{t} - \lambda : \dom (D_t) \ra \LLL^2 \big ( \Lambda ^*( T^*(X \setminus \Sigma )) \big )  $, where $\lambda \in \C$, $| \lambda| \in [ e^{-ct/2} , \frac{C \sqrt{t} }{2} ]$, with constants $c,C$ as in Proposition \ref{estimatea1}. We deduce the invertibility of the operator $D_t - \lambda$, and since $\Delta _t - \lambda ^2 = (D_t- \lambda) (D_t + \lambda)$ we thus get Part (1) of  Theorem \ref{thmspectralgap}. \hfill $\Box$

{\it Proof of Part (2) of Theorem \ref{thmspectralgap} and Corollary \ref{cormorse}}:

For $i = 0,\ldots, 2 \nu $ we define the $\RRR$-vector space $C^i$ by \begin{equation} C^i: = \bigoplus _{p \in \Crit _i (f)\setminus \Sigma } \RRR \cdot e_{p,1}^i \oplus \bigoplus _{p \in  \Sigma, j \in I_p^i}\RRR \cdot e^i_{p,j}. \end{equation}   Note that by construction $\dim C^i = c_i (f) $ with $c_i (f)$ as in \eqref{mpi1}. We define a linear map \begin{equation} J_i(t) :   C^ i   \longrightarrow  \CCC _{t}^i, \ J_i(t)(e_{p,j}^i) =   \Phi_{ p,j}^i (t) .\end{equation} We denote  by $(\FFF _t, d_t, \langle \ , \ \rangle  ) $ the subcomplex of $ ({\CCC _t}, d_t , \langle \ , \ \rangle )$ generated by the eigenforms of   $\Delta _t$ to  eigenvalues lying in $ [0,1]$. We denote moreover by $P (t,[0,1])$ the orthogonal projection operator from $\CCC _t$ on $\FFF _t$ with respect to $\langle \ , \ \rangle $.

 As in \cite{ulcurve}, Section 4, one can now show that the linear map $P_i ( t,[0, 1]) \circ  J_i (t) :  C ^i  \longrightarrow \FFF _{t} ^i$ is a bijective map from $C ^i$ onto $\FFF _{t} ^i$ and thus  the complex  $(\FFF _t, d_t, \langle \ , \ \rangle )$  is a finite dimensional subcomplex of $( \CCC _t, d_t , \langle \ , \ \rangle )$ with $\dim \FFF _{t} ^i = c_i(f).$  By Proposition \ref{propstokes}  moreover $H^*_{(2)} (X) \simeq \ker (\Delta _t )\simeq H^*( \FFF _t, d_t , \langle \ , \ \rangle ).$ The Morse inequalities in Corollary \ref{cormorse} now follow by a standard algebraic argument. \hfill $\Box$

\section*{Acknowledgements}\noindent I wish to thank Jean-Michel Bismut for  suggesting work on the subject. I am grateful to Jochen Br\"uning and Jean-Paul Brasselet for many helpful conversations, as well as to Olaf Post. A part of this work  was done while I was supported by SFB 647.

\bibliographystyle{amsplain}

\end{document}